\RequirePackage{lineno}

\documentclass[11pt]{amsart}

\usepackage{amsthm}
\usepackage{amsmath}
\usepackage{amssymb}
\usepackage[margin=1in]{geometry}
\usepackage{enumerate}
\usepackage{mathrsfs}  

\usepackage{hyperref}

\usepackage[numbers]{natbib}

\title{A volume-based approach to the Multiplicative Ergodic Theorem on Banach spaces}
\author{Alex Blumenthal}

\address{Courant Institute of Mathematical Sciences\\
New York University\\
New York, NY 10012, USA}
\email[A. Blumenthal]{alex@cims.nyu.edu}


\theoremstyle{theorem}
\newtheorem{thm}{Theorem}[section]
\newtheorem{cor}[thm]{Corollary}
\newtheorem{lem}[thm]{Lemma}
\newtheorem{prop}[thm]{Proposition}

\theoremstyle{definition}

\newtheorem{defn}[thm]{Definition}
\newtheorem{rmk}[thm]{Remark}

\newtheorem{cla}[thm]{Claim}

\numberwithin{equation}{section}

\newcommand{\N}{\mathbb{N}}

\newcommand{\R}{\mathbb{R}}

\newcommand{\Fc}{\mathcal{F}}

\newcommand{\Gc}{\mathcal{G}}

\renewcommand{\a}{\alpha}

\renewcommand{\d}{\delta}
\newcommand{\e}{\epsilon}
\renewcommand{\l}{\lambda}

\newcommand{\Id}{\text{Id}}

\newcommand{\Bc}{\mathcal{B}}

\renewcommand{\tilde}{\widetilde}

\newcommand{\ds}{/ \! /}

\newcommand{\Sc}{\mathcal{S}}

\newcommand{\Oc}{\mathcal{O}}

\newcommand{\ol}{\overline{\lambda}}
\newcommand{\ul}{\underline{\lambda}}

\newcommand{\codim}{\operatorname{codim}}

\begin{document}

\maketitle

\begin{abstract}
A volume growth-based proof of the Multiplicative Ergodic Theorem for Banach spaces is presented, following the approach of Ruelle for cocycles acting on a Hilbert space. As a consequence, we obtain a volume growth interpretation for the Lyapunov exponents of a Banach space cocycle.
\end{abstract}

\section{Introduction and Statement of Results}

\subsection{Introduction}

The purpose of this paper is to present a volume-based proof of the Multiplicative Ergodic Theorem (MET) for cocycles on a Banach space, from which a volume growth rate interpretation of Lyapunov exponents for such cocycles will be readily apparent.

By now, there are many proofs of the MET for cocycles on finite dimensional vector spaces, e.g., \cite{O, Ra, R1, W}; for additional treatments, see Section 3.4 of \cite{A} or \S 1.5 of \cite{Kr}. For applications, particularly for the purposes of smooth ergodic theory, the most successful interpretation of Lyapunov exponents is that of volume growth: the exponential growth rate of the $q$-dimensional volume of a `generic' $q$-dimensional subspace is the sum of the first $q$ Lyapunov exponents.


The volume growth interpretation for the Lyapunov exponents for cocycles of operators on finite dimensional spaces is implicit in Ruelle's proof of the MET \cite{R1}: the $q$-th Lyapunov exponent is defined as the exponential growth rate of the $q$-th singular value of high iterates of the cocycle, and it is shown that the `slow' growing singular value decomposition subspaces for high iterates of the cocycle converge at an exponentially fast rate to the subspaces of `slow' growing vectors for the cocycle. The very same approach is used to prove the MET for compact cocycles on Hilbert spaces in \cite{R2}, and the volume growth interpretation is the same, since the finite dimensional subspaces of a Hilbert space carry notions of volume and determinant inherited from the inner product on the ambient Hilbert space. 

For cocycles on Banach spaces, the MET was originally proved by Ma{\~n}{\'e} for compact, injective cocycles satisfying a certain continuity condition \cite{M}.  Thieullen extended Ma{\~n}{\'e}'s result to include injective cocycles satisfying a quasicompactness condition \cite{T}, and Lian and Lu proved a version of the MET for injective cocycles on separable Banach spaces satisfying only a measurability condition \cite{LL}. The method in these papers is to show that the `fast' growing subspaces (that is, subspaces of vectors achieving a sufficient rate of exponential contraction in backwards time) are almost surely finite dimensional.  Then, Lyapunov exponents are defined to be the exponential rates at which the dimension of the `fast' growing subspaces has a jump discontinuity, and it is shown that the `slow' growing subspaces are the graphs of mappings given by a certain convergent infinite series of graphing maps.

Thieullen also considered non-injective cocycles \cite{T}, and deduced from the MET for injective cocycles that even when the cocycles aren't injective, the `slow' growing subspaces are finite-codimensional and vary measurably, and the only possible growth rates of vectors are  Lyapunov exponents; this result is referred to as the One-Sided MET, and holds even when the base dynamics are noninvertible. Thieullen's proof in \cite{T} involves passing to an injective `natural extension' cocycle over an augmented Banach space. Doan, in his PhD thesis \cite{Doan}, carried out an analogous deduction for strongly measurable cocycles on a separable Banach space, assuming the MET in \cite{LL}.

Quas, Froyland and Lloyd \cite{QLF} deduced from the One-Sided MET in \cite{T} that even when the cocycle is not injective, there is still an invariant, measurably-varying distribution of `fast' growing subspaces, as long as the base dynamics are invertible. Quas and Gonz{\'a}lez-Tokman \cite{QT} also deduced the existence of an invariant distribution of `fast'-growing subspaces for cocycles on a separable Banach space satisfying the measurability condition in \cite{LL}, assuming the One-Sided MET in \cite{Doan}.

In all known proofs of the Banach space MET, genuine finite-dimensional volumes (with multiplicative determinants) are not used, and there is no apparent volume growth interpretation of the Lyapunov exponents. So, it is natural to ask whether there can be a volume growth interpretation for Lyapunov exponents of cocyles acting on a Banach space, noting that there is no trouble assigning volumes to the finite dimensional subspaces of a Banach space, using Haar measure, and that one can define determinants as volume ratios taken using these volumes. 



Indeed, once the MET is proven by other means, it is possible to show that Lyapunov exponents have the very same volume growth interpretation as in the Hilbert space case. In light of Ruelle's approach to the MET in \cite{R1, R2}, though, this order of actions is  strange: volume growth is used to prove the MET, and the volume growth interpretation of the Lyapunov exponents follows naturally.

The purpose of the present paper is to show how to deduce the Banach space MET using volume growth ideas, using arguments analogous to those in \cite{R2}, from which the volume growth rate interpretation of Lyapunov exponents will be easily seen. The proof we present treats non-injective cocycles with no added effort, and yields immediately that vectors sufficiently inclined away from `slow' growing subspaces grow uniformly quickly.

\medskip

After the work in the present manuscript was completed, the author was made aware of the similar work of A. Quas and C. Gonzalez-Tokman \cite{gonzalez2014concise}. The present work was carried out simultaneously with and independently of \cite{gonzalez2014concise}. We note that there are several technical differences between the precise statement of the MET in \cite{gonzalez2014concise} and that given in the present paper. 

\medskip

For the rest of this section, we will define and describe the volumes on finite dimensional subspaces that we use, state assumptions on the cocycles we will study, and state the MET. In Section 2, we will give preliminaries on the geometry of Banach spaces and the volumes on finite dimensional subspaces necessary for the proof, which will be given in Section 3. 

\subsection{Statement of Results}

We now state our results precisely, beginning with an overview of volumes on finite dimensional subspaces of a Banach space.

\subsubsection{Volumes and Determinants on Finite Dimensional Subspaces}

We let $(\Bc, |\cdot|)$ be an infinite dimensional Banach space, and let $E \subset \Bc$ be a a finite dimensional subspace. Treating $E$ as a topological group with the group action given by addition of vectors, there exists a nonzero, translation invariant measure on $E$, unique up to a scalar factor: the Haar measure on $E$ (\S2.2 of \cite{F}). In the definition of volumes given below, we normalize the Haar measure on $E$ by assigning a particular volume to the $|\cdot |$- unit ball of $E$.


\begin{defn}\label{defn:indVol}
Let $q \in \N$ and let $E \subset \Bc$ be a $q$-dimensional subspace. We define the \emph{induced volume} $m_E$ on $E$ to be the unique Haar measure on $E$ for which
\[
m_E \{e \in E : |e| \leq 1\} = \omega_q,
\]
where $\omega_q$ is the $q$-dimensional volume of the Euclidean unit ball in $\R^q$.
\end{defn}

Many identities involving volume on inner product spaces can be recovered for the induced volume as inequalities, with uniform bounds depending only on the dimension $q$ of the subspace. The following is a prototypical example, and demonstrates that the induced volume really does `see' the geometry on finite dimensional subspaces.

\begin{lem}\label{lem:sampleVol}
If $E \subset \Bc$ has dimension $q < \infty$, then for any $v_1, \cdots, v_q \in E$, we have
\[
 m_E P[v_1, \cdots, v_q] \approx |v_q| \cdot \prod_{i = 1}^{q-1} d(v_i, \langle v_j : i < j \leq q \rangle), 
\]
where $P[v_1, \cdots, v_q] = \{c_1 v_1 + \cdots + c_q v_q : 0 \leq c_i \leq 1, 1 \leq i \leq q\}$ is the parallelepiped spanned by $v_1, \cdots, v_q$. For vectors $w_1, \cdots, w_q \in \Bc$, we denote by $\langle w_1, \cdots, w_k \rangle$ the subspace of $\Bc$ spanned by the vectors $w_1, \cdots, w_q$.  The symbol $\approx$ denotes equality up to a multiplicative constant depending only on $q$.
\end{lem}
\noindent Using the tools developed in Section 2, the proof of Lemma \ref{lem:sampleVol} is straightforward, and is sketched in Remark \ref{rmk:sampleVolProof}.

With these volumes in place, the determinant is defined as a volume ratio: for a bounded linear map $A : \Bc \to \Bc$ and a finite dimensional subspace $ E \subset \Bc$, we write
\begin{equation} \label{eq:detDefinition}
\det(A | E) =  \begin{cases}
 \frac{m_{AE} A \Oc}{m_E \Oc} & A \text{ injective on } E, \\
 0 & \text{otherwise,}
 \end{cases}
\end{equation}
where $\Oc$ is any Borel subset of $E$ with nonzero $m_E$-volume. By the uniqueness of Haar measure, the right-hand side of \eqref{eq:detDefinition} does not depend on the set $\Oc \subset E$.

Define the maximal $q$-dimensional volume growth
\begin{equation}\label{eq:maxVolGrow}
V_q(A) = \sup\{\det(A  |E) : E \subset \Bc, \dim E = q\}
\end{equation}
for bounded, linear maps $A : \Bc \to \Bc$, and observe that since the determinant is multiplicative, being a Radon-Nikodym derivative, $V_q$ is submultiplicative, i.e., for bounded linear maps $A, B : \Bc \to \Bc$, we have $V_q(AB) \leq V_q(A) V_q(B)$.

\subsubsection{Statement of the MET}

Let $(X, \Fc, \mu)$ be a probability space, and let $f$ be a measure preserving transformation on $(X, \Fc, \mu)$. Throughout the paper, we assume for simplicity that $\mu$ is ergodic; the only real difference in the nonergodic case is that the Lyapunov exponents and their multiplicities may now depend on the point in the base space, but this complication does not change the underlying arguments in any substantive way.

A cocycle is a map $T : X \to L(\Bc)$, where $L(\Bc)$ is the space of bounded linear operators on $\Bc$, subject to the following measurability assumption.
\begin{align}\label{eq:measCondition}
\begin{split}
\text{There is a sequence of finite-valued} \text{ maps $T^{(n)} : X \to L(\Bc)$ such that for} \\
\text{any $x \in X$, $T^{(n)}_x \to T_x$ in } L(\Bc)  \text{ in the induced operator norm as $n \to \infty$.}
\end{split}
\end{align}
We write $T_x, T^{(n)}_x$ for $T, T^{(n)}$ evaluated at $x \in X$, respectively. The condition \eqref{eq:measCondition} is called \emph{uniform measurability} (Definition 3.5.5 in \cite{Hille}).

We write $T^n_x = T_{f^{n-1} x} \circ \cdots \circ T_x$, and for $\l \in \R$, we define the subspaces
\begin{align}\label{eq:slowGrowSpaces}
F_{\l}(x) = \{v \in \Bc : \limsup_{n \to \infty} \frac{1}{n} \log |T^n_x v| \leq \l \},
\end{align}
using the convention $\log 0 = - \infty$. For $A \in L(\Bc)$, we denote by $|A|_{\a}$ the measure of noncompactness of $A$, which will be defined precisely in Definition \ref{defn:measNoncompactness}.
 
%
We now state the main result of this paper, a version of the MET emphasizing the volume growth interpretation of the Lyapunov exponents of a Banach space cocycle. Note that we do not assume that $T_x$ is an injective operator.

\begin{thm} \label{thm:volMET}
Let $T: X \to L(\Bc)$ satisfy \eqref{eq:measCondition}, and assume moreover that
\[\int_X \log^+ |T_x| \; d \mu(x) < \infty.\]
Then, there is a $\mu$-full, $f$-invariant set $\Gamma \subset X$ with the following properties.
\begin{enumerate}
\item For $x \in \Gamma$, the $q$-dimensional growth rates
\[
l_q = \lim_{n \to \infty} \frac{1}{n} \log V_q(T^n_x)
\]
exist and are constant over $\Gamma$. Similarly, $l_{\a} = \lim_{n \to \infty} \frac{1}{n} \log |T^n_x|_{\a}$ exists and is constant over $\Gamma$.
\item Defining the sequence $\{K_q\}_{q \geq 1}$ by $K_1 = l_1$ and $K_q = l_q - l_{q - 1}$ for $q > 1$, we have that $K_1 \geq K_2 \geq \cdots$, and that $K_q \to l_{\a}$ as $q \to \infty$. We denote by $\l_i$ the distinct values of the sequence $\{K_q\}_q$ with finite multiplicities $m_i$; in particular, for any $l > l_\a$, we have that $\sum_{i : \l_i \geq l} m_i < + \infty$. The Lyapunov exponents $\l_i$ satisfy the following.
\begin{enumerate}
\item For any Lyapunov exponent $\l_i$, we have that $\codim F_{\l_i}(x) = m_1 + m_2 + \cdots + m_{i -1}$.
\item For any subsequent pair of Lyapunov exponents $\l_i > \l_{i + 1} \geq l_{\a}$ and for any $x \in \Gamma$, $v \in F_{\l_i}(x) \setminus F_{\l_{i + 1}}(x)$,
\[
\lim_{n \to \infty} \frac{1}{n} \log |T^n_x v| = \l_i.
\]
\item (Volume Growth) For any $x \in \Gamma$, any Lyapunov exponent $\l_i$, and any complement $E$ to $F_{\l_i}(x)$, we have that
\begin{align}\label{eq:correctVolGrow}
\lim_{n \to \infty} \frac{1}{n} \log \det(T^n_x | E) = m_1 \l_1 + \cdots + m_i \l_i.
\end{align}
\end{enumerate}
\item The mapping $x \mapsto F_{\l_i}(x)$ is measurable.
\end{enumerate}
\end{thm}

\begin{rmk}
The measurability assumption on the cocycle in \cite{M, T} is called $\mu$-\emph{continuity}, and is stronger than \eqref{eq:measCondition}: $\mu$-continuity requires the base space $X$ to be a compact topological space with Borel $\sigma$-algebra $\Fc$, the base map $f$ to be continuous, and the existence of a sequence of disjoint compact sets $\{K_n\}_n$ in $X$ with $\mu \left( \cup_n K_n \right) = 1$ for which $T|_{K_n}$ is norm continuous $K_n \to L(\Bc)$. 

In \cite{LL}, the Banach space $\Bc$ is assumed separable, and the measurability condition in that paper is \emph{strong measurability}, which requires that the evaluation maps $x \mapsto T_x v$ for $v \in \Bc$ be measurable as maps $(X, \Fc) \to (\Bc, \operatorname{Bor}(\Bc))$ (Definition 3.5.4 in \cite{Hille}), where $\operatorname{Bor}(\Bc)$ refers to the Borel sigma-algebra for the norm topology on $\Bc$; the same measurability hypothesis is used in \cite{gonzalez2014concise}. The condition \eqref{eq:measCondition} implies strong measurability (see Corollary 1 to Theorem 3.5.3 in \cite{Hille}).
\end{rmk}

\begin{rmk}
Item (3) in Theorem \ref{thm:volMET} is formulated precisely and proved in the Appendix as Lemma \ref{lem:measurabilitySubspace}. We note that this measurability is entirely independent of the arguments proving items (1), (2); however, we include it for the sake of completeness.
\end{rmk}



\subsection{Notation}
We collect here the various notations used in the paper. 

\begin{enumerate}

\item We denote by $(X, \Fc, \mu, f)$ an ergodic m.p.t.; we do not assume that $f$ is invertible.

\item For $a,b \in \R$, $a \vee b = \max\{a, b\}$.

\item Let $q \in \N$. For $a,b \in [0,\infty)$, we will write $a \lesssim b$ when $a \leq C_q b$, and $a \approx b$ if $a \lesssim b$ and $b \lesssim a$. It will be clear from context which constant $\lesssim, \approx$ depend on.

\item For normed vector spaces $V_1, V_2$, we denote by $L(V_1, V_2)$ the space of bounded linear operators $V_1 \to V_2$. 

\item $\Bc$ will always refer to an infinite dimensional Banach space with norm $|\cdot|$, and we always refer to the induced operator norm on $L(\Bc) := L(\Bc, \Bc)$ by the same symbol $|\cdot|$.

\item For a subspace $E \subset \Bc$, we let $B_E := \{v \in E : |v| \leq 1\}$ denote the closed unit ball and $S_E := \{v \in E : |v| = 1\}$ denote the unit sphere. 
\item For vectors $v_1, \cdots, v_m \in \Bc$, we denote by $\langle v_i : 1 \leq i \leq m \rangle$ or $\langle v_1, \cdots, v_m \rangle $ the subspace of $\Bc$ spanned by $v_1, \cdots, v_m$.  We denote the parallelepiped spanned by $v_1, \cdots, v_m$ by $P[v_1, \cdots, v_m] = \{c_1 v_1 + \cdots + c_m v_m : c_i \in [0,1]\}$.

\item Let $A \in L(\Bc)$ and let $E \subset \Bc$ be a subspace. We denote by $A|_E \in L(E, \Bc)$ the restriction of $A$ to $E$. 
\item When $\dim E < \infty$, we denote by $\det(A | E)$ the determinant of $A$ on $E$, as defined in \eqref{eq:detDefinition}.

\item For a subspace $F \subset \Bc$, we denote by $F^{\circ}$ the annihilator of $F$ in $\Bc^*$, i.e., 
\[F^{\circ} := \{l \in \Bc^* : l(f) = 0 \text{ for all } f \in F\}.\]
The codimension of a closed subspace $F \subset \Bc$ is defined by $\codim F = \dim F^{\circ}$.

\item For a vector space $V$, we write $V = E \oplus F$ when $V$ is the direct sum of $E$ and $F$, i.e., each vector in $V$ can be written as the unique sum of an element of $E$ and an element of $F$, and we call such a splitting of $V$ an \emph{algebraic splitting}; we write $\pi_{E \ds F}$ for the projection operator onto $E$ parallel to $F$. 

\item When $V$ is a normed vector space, $V = E \oplus F$ is an algebraic splitting, and $\pi_{E \ds F}$ is a bounded linear operator, we call $V = E \oplus F$ a \emph{topological splitting}. For a closed subspace $E \subset \Bc$, if $F \subset \Bc$ is a closed subspace for which $\Bc = E \oplus F$ is a topological splitting, we call $F$ a \emph{topological complement} to $E$.

\item $\Gc(\Bc)$ refers to the Grassmanian of closed subspaces of $\Bc$, and $\Gc_q(\Bc), \Gc^q(\Bc)$ refer to the subsets of $q$-dimensional and $q$-codimensional subspaces of $\Bc$, respectively. $d_H$ refers to the Hausdorff distance on $\Gc(\Bc)$, defined in \eqref{eq:hausDist}.

\item For $A \in L(\Bc)$ and $q \in \N$, $c_q(A)$ denotes the $q$-th Gelfand number of $A$, defined in \eqref{eq:gelfand}, and $c(A)$ is defined as the limit $\lim_{q \to \infty} c_q(A)$. We denote by $|A|_{\a}$ the measure of compactness of $A$, as in Definition \ref{defn:measNoncompactness}.

\end{enumerate}

\section{Induced Volumes and Banach Space Geometry}

In this section, we collect some results on the geometry of Banach spaces and discuss the geometric properties of the induced volumes defined in Section 1.

For our purposes, there are two echelons for the geometry for Banach spaces: the `local' geometry of finite dimensional subspaces and the induced volumes on them, and the `global' geometry of the Banach space at large.
 

\subsection{Global Geometry of Banach Spaces}

We begin by defining a concept of `angle' in Banach space, and showing its relation to the projection maps associated to splittings. Next, we recall some facts about the geometry of the Grassmanian of closed subspaces of $\Bc$. Finally, we recall the definition and some basic properties of the measure of noncompactness $|\cdot|_{\a}$.

\subsubsection{Angles, Splittings and Complements}
Key to volume-growth based approaches to the MET is a concept of angle between vectors and subspaces. The following notion of angle is well-suited to our purposes. Definition \ref{defn:angle} is adapted from Section II of Ma{\~n}{\'e}'s paper \cite{M}, and is also very similar to the minimal angle defined in Section 5 of \cite{Berkson}.


\begin{defn}\label{defn:angle}
Let $E, F \subset \Bc$ be subspaces. The \emph{minimal angle} $\theta(E, F) \in [0,\frac{\pi}{2}]$ from $E$ to $F$ is defined by
\[
\sin \theta (E, F) = \inf\{|e - f | : e \in E, |e| = 1, f \in F\}.
\]
\end{defn}
Roughly speaking, the minimal angle $\theta(E, F)$ will be small whenever $E$ is inclined towards $F$. Note that $\theta(E, F)$ may not equal $\theta(F, E)$. The minimal angle has an interpretation in terms of splittings of Banach space.

\begin{lem} \label{lem:angleForm}
Let $\Bc = E \oplus F$ be an algebraic splitting, and assume $E \neq \{0\}$. Then, $\Bc = E \oplus F$ is topological if and only if $\theta(E, F) > 0$, in which case we have
\[
\sin \theta(E, F) = |\pi_{E \ds F}|^{-1},
\]
where $\pi_{E \ds F}$ is the projection to $E$ parallel to $F$.
\end{lem}

\begin{proof}
Let $P = \pi_{E \ds F}$, which is well-defined and possibly unbounded when $\Bc = E \oplus F$ is not topological. One computes
\begin{align*}
|P| &= \sup_{e \in E, f \in F, e + f \neq 0} \frac{|e|}{|e + f|} = \sup_{e \in E, |e| = 1, f \in F} \frac{1}{|e + f|} = \frac{1}{\inf\{|e - f | : e \in E, |e| = 1, f \in F\}} \\
& = \frac{1}{\sin \theta(E, F)},
\end{align*}
where we interpret the RHS of this formula as $\infty$ whenever $\theta(E, F) = 0$.
\end{proof}



It is crucial to also have access to complements for subspaces of finite dimension and codimension, as there is no assignment of an `orthogonal complement' to closed subspaces of a Banach space.



\begin{lem}[III.B.10 and III.B.11 of \cite{Wo}] \label{lem:compExist}
Fix $q \in \N$.

For any subspace $E \subset \Bc$ with $\dim E = q$, there exists a topological complement $F$ for $E$ such that $\sin \theta(E, F) \geq \frac{1}{\sqrt{q}}$, i.e., $|\pi_{E \ds F}| \leq \sqrt{q}$.

For any closed subspace $F \subset \Bc$ with $\codim F = q$, there exists a topological complement $E$ for $F$ such that $\sin \theta(F, E) \geq \frac{1}{\sqrt{q} + 2}$, i.e., $|\pi_{F \ds E}| \leq \sqrt{q} + 2$.
\end{lem}

We will also need the following lemma regarding complements.
%

\begin{lem}\label{lem:splittingExist}
Let $A \in L(\Bc)$ and $E_1 \subset \Bc$ be a finite dimensional subspace for which $A|_{E_1}$ is injective. If $F_2$ is a topological complement to $E_2 := A E_1$, then $F_1 = \{v \in \Bc : A v \in F_2\}$ is a $\dim E_1$-codimensional subspace complementing $E_1$. Moreover,
\begin{align}\label{eq:projectionForm}
\pi_{E_1 \ds F_1} = (A|_{E_1})^{-1} \circ \pi_{E_2 \ds F_2} \circ A
\end{align}

\end{lem}
\begin{proof}
Let $P = (A|_{E_1})^{-1} \circ \pi_{E_2 \ds F_2} \circ A$, which is a bounded linear operator on $\Bc$ with range $E_1 \subset \Bc$. One computes that $P|_{E_1}$ is the identity on $E_1$, and so $P$ is a bounded projection operator.  Define $\tilde{F} = \ker P$, which is a topological complement to $E_1$. With $F_1 := \{v \in \Bc : A v \in F_2\}$, we will show that $\tilde{F} = F_1$, as desired. 

First, let $f_1 \in F_1$. The definition of $P$ implies $A \circ P f_1 = \pi_{E_2 \ds F_2} \circ  A f_1$, but $A f_1 \in F_2 = \ker \pi_{E_2 \ds F_2}$ by assumption, and so $A \circ P f_1 = 0$. But we assumed that $A|_{E_1}$ is injective, and so $P f_1 = 0$; so, we have deduced that $\tilde{F} \supset F_1$. For the other inclusion, if $f \in \tilde{F}$, then $P f = 0$ implies $A f \in \ker \pi_{E_2 \ds F_2} = F_2$, and so $f \in \{v \in \Bc : A v \in F_2\} = F_1$ follows, hence $ \tilde{F} \subset F_1$.
\end{proof}

\subsubsection{The Grassmanian on Banach Space}

We denote by $\Gc(\Bc)$ the space of closed subspaces of $\Bc$, endowed with the Hausdorff distance of unit spheres: for $U, V \in \Gc(\Bc)$, we set
\begin{align}\label{eq:hausDist}
d_H(U, V) = d_H(S_U, S_V),
\end{align}
where $S_U = \{u \in U : |u| = 1\}$, and the Hausdorff distance between closed subsets $A, B$ of a metric space is defined by
\begin{align*}
d_H(A, B) = \max\{\sup_{a \in A} d(a, B), \sup_{b \in B} d(b, A) \},
\end{align*}
where $d(a, B) := \inf\{d(a,b) : b \in B\}$. We define $\Gc_q(\Bc) = \{E \in \Gc(\Bc) : \dim E = q\}$ and $\Gc^q(\Bc) = \{F \in \Gc(\Bc) : \codim F = q\}$. The following are some basic properties of the Hausdorff distance $d_H$.

\begin{prop}[Chapter IV, \S2.1 of \cite{K}] \label{prop:grassProps}
The metric space $(\Gc(\Bc), d_H)$ has the following properties.
\begin{enumerate}
\item $(\Gc(\Bc), d_H)$ is complete. 
\item The subsets $\Gc_q(\Bc), \Gc^q(\Bc)$ are closed in $(\Gc(\Bc), d_H)$.
\end{enumerate}
\end{prop}

It is actually somewhat inconvenient to compute directly with $d_H$, and so frequently it will be easier to use the \emph{gap}, defined for $E, E' \in \Gc(\Bc)$ by
\[
\d(E, E') := \sup_{e \in E, |e| = 1} d(e, E').
\]
\begin{lem} \label{lem:apertureProps}
Let $U, V \in \Gc(\Bc)$.
\begin{enumerate}
\item $\d(U, V) \vee \d(V, U) \leq d_H(U, V) \leq 2 (\d(U, V) \vee \d(V, U))$.
\item $\d(U, V) = \d(V^{\circ}, U^{\circ})$.
\item Let $q \in \N$, and assume either that $U, V \in \Gc_q(\Bc)$ or that $U, V \in \Gc^q(\Bc)$. If $\d(U, V) < \frac{1}{q}$, then
\begin{align}\label{eq:gapEstimate}
\d(V, U) \leq \frac{q \d(U, V)}{1 - q \d(U, V)}.
\end{align}
\end{enumerate}
\end{lem}
\begin{proof}
Item 1 is proven in Chapter IV, \S2.1 of \cite{K}, and Item 2 is Lemma 2.8 in Chapter IV, \S 2.3 of \cite{K}. For Item 3, if $U, V \in \Gc_q(\Bc)$, let $\{u_1, \cdots, u_q\}$ be a basis of $U$: for any $\a > 1$, we can obtain a basis $\{v_i\}_{i = 1}^q$ of $V$ by demanding that $|u_i - v_i| \leq \a \d(V, U)$. We can now compare $v = c_1 v_1 + \cdots + c_q v_q$ with $u := c_1 u_1 + \cdots + c_1 u_q$, proving \eqref{eq:gapEstimate} in this case. For $U, V \in \Gc^q(\Bc)$, note that by Item 2, $\d(U, V) = \d(V^{\circ}, U^{\circ})$ and $\d(V, U) = \d(U^{\circ}, V^{\circ})$, and so we can apply the $q$-dimensional version of \eqref{eq:gapEstimate} to conclude in this case.
\end{proof}

When $\Bc = E \oplus F$ is a topological splitting, it is important to know when a subspace $E' \subset \Bc$ nearby to $E$ is also complemented to $F$. Using intuition from finite dimensional geometry, it is reasonable to suggest that this involves the distance $d_H(E, E')$ and the inclination of $E$ towards $F$, namely $\sin \theta(E, F)$.

\begin{prop} \label{prop:openCond}
Let $\Bc = E \oplus F$ be a topological splitting. For any $E' \in \Gc(\Bc)$ for which $d_H(E, E') < \sin \theta(E, F)$, we have that $\Bc = E' \oplus F$ is a topological splitting, in which case the projection $\pi_{F \ds E'}$ satisfies
\begin{align}\label{eq:graphNormEst}
|\pi_{F \ds E'}|_E| \leq \frac{2 d(E, E')}{|\pi_{E \ds F}|^{-1} - d(E, E')}.
\end{align}
\end{prop}
The proof of Proposition \ref{prop:openCond} is given in full detail the Appendix. 

\subsubsection{Measure of Noncompactness}

We close this section with a brief review of the measure of noncompactness $|\cdot|_{\a}$. Let $C \subset \Bc$ be a bounded set. We define $q(C)$ to be the infimum over the set of all $r > 0$ for which $C$ admits a finite open cover by $|\cdot|$-balls of radius $r$.

\begin{defn}\label{defn:measNoncompactness}
Let $\Bc, \Bc'$ be two Banach spaces, and let $A \in L(\Bc, \Bc')$. We define the measure of noncompactness of $A$ to be $|A|_{\a} := q(A B_{\Bc})$.
\end{defn}

$|\cdot|_{\a}$ and related notions are discussed in \cite{Akhmerov}. Definition \ref{defn:measNoncompactness} is actually given as Theorem 2.4.2 in \cite{Akhmerov}, as $|\cdot|_{\a}$ is defined there in a different, yet equivalent, way. We now list some of the properties of $|\cdot|_{\a}$.

\begin{prop}[1.4.1 of \cite{Akhmerov}] \label{prop:kuratowski}
$|\cdot|_{\a}$ is a submultiplicative seminorm on $L(\Bc, \Bc')$. Precisely, let $\Bc, \Bc', \Bc''$ Banach spaces; for $T, U \in L(\Bc, \Bc')$ and $V \in L(\Bc', \Bc'')$, we have
\begin{gather*}
|T + U|_{\a} \leq |T|_{\a} + |U|_{\a}, \\
|V U|_{\a} \leq |V|_{\a} \cdot |U|_{\a}.
\end{gather*}
Moreover, $T$ is compact iff $|T|_{\a}= 0$.
\end{prop}

\subsection{Finite dimensional subspaces and the induced volume}

In this part, we collect results on the induced volume in Definition \ref{defn:indVol}. 

We begin this subsection by enumerating the most basic properties of the induced volume. Second, we will approximate the norm on a given finite dimensional subspace by an inner product, using John's Theorem, obtaining access to the `exact' identities and relations between volumes and angles for inner product spaces. Translating what we do in the inner product setting back into the normed vector space vocabulary, we will obtain estimates with errors depending only on the dimension. Third, we will consider the measurability properties of the maximal $q$-dimensional growths $V_q$ defined in \eqref{eq:maxVolGrow}. Fourth, we will relate $V_q$ with the measure of noncompactness $|\cdot|_{\a}$.

\subsubsection{Basic Properties of the Induced Volume and Determinants}

Recall from Section 1 that the induced volume $m_E$ for a $q$-dimensional subspace $E \subset \Bc$ is the Haar measure on $E$ normalized so that $m_E \{v \in E : |v| \leq 1\} = \omega_q$, the mass of the unit ball in $\R^q$. We now give the most basic properties of $m_E$. 

\begin{lem}\label{lem:indVol}
Let $q \in \N$ and let $E \subset \Bc$ be a $q$-dimensional subspace. The induced volume $m_E$ satisfies the following.
\begin{enumerate}
\item For any $v \in E$ and any Borel $B \subset E$, we have $m_E(v + B) = m_E (B)$.
\item If $m'$ is any other non-zero, translation-invariant measure on $E$, then $m', m_E$ are equivalent measures. The Radon-Nikodym derivative $\frac{d m'}{dm_E}$ is constant on $E$, and equals $\frac{m'(B)}{m_E(B)}$, where $B \subset E$ is any Borel set with positive $m_E$-measure.
\item For any $a > 0$ and any Borel measurable set $B \subset E$, we have $m_E (a B) = a^q m_E(B)$.
\item Let $w_1, \cdots, w_q$ is any set of vectors in $E$, and write $P[v_1,\cdots,v_q] = \{\sum_{i = 1}^q \l_i v_i : 0 \leq \l_i \leq 1\}$ for the parallelepiped spanned by $\{v_1, \cdots, v_q\}$. Then, for any $a > 0$,
\[
m_E P[a w_1, w_2, \cdots, w_q] = a \cdot m_E P[w_1, \cdots, w_q].
\]

\end{enumerate}
\end{lem}
\begin{proof}
These all follow from the uniqueness of Haar measure \cite{F}, except for Item 4, which has a simple proof following from Proposition \ref{prop:johnThm}.
\end{proof}

\begin{rmk}
The normalization in Definition \ref{defn:indVol} for the induced volume was introduced by Busemann \cite{B} as a way of assigning volume elements to Finsler manifolds, where the tangent spaces are normed vector spaces, not inner product spaces. For more on this, see Chapter I, \S 8 of \cite{Ru}.
\end{rmk}

Recall now the definition of the determinant.
\begin{defn}\label{defn:det}
Let $A \in L(\Bc, \Bc')$ be a map of Banach spaces, and let $E \subset \Bc$ be a finite dimensional subspace. We define the determinant $\det(A | E)$ of $A$ on $E$ by
\[
\det (A|E) = 
\begin{cases}
\frac{m_{A E}(A B_E)}{m_E(B_E)} & A|_E \text{ injective,} \\
0 & \text{otherwise.}
\end{cases}
\]
\end{defn}

The following are some basic properties of this determinant.
\begin{prop}[Basic Properties of the Determinant] \label{prop:detProps}
Let $E, F, G$ be normed vector spaces of the same finite dimension, and let $A : E \to F, B : F \to G$ be linear maps.
\begin{enumerate}
\item For any Borel set $\Oc \subset E$ with $m_E (\Oc) > 0$,
$$
\det (A|E) = \frac{m_F(A \Oc)}{m_E(\Oc)},
$$
and so $\det(A | E)$ coincides with the Radon-Nikodym derivative $\frac{d (m_F \circ A)}{d m_E}$, where $(m_F \circ A)(\Oc) := m_F(A \Oc)$ for Borel $\Oc \subset E$.
\item If $A$ is not injective on $E$, then $\det (A | E) = 0$.
\item $\det (B A | E) = \det (B| F) \cdot \det (A | E)$.
\end{enumerate}
\end{prop}
\begin{proof}
All these follow, once again, from the uniqueness of Haar measure.
\end{proof}

\subsubsection{Approximation by Inner Products}

Every pair of norms on a finite dimensional vector space is equivalent, and so given any finite dimensional subspace $E \subset \Bc$, one can always compare the norm on $E$ with \emph{any} choice of inner product $(\cdot, \cdot)$ on $E$. There is no control, however, on how bad this approximation can be, and so we must take care to ensure our approximating inner product is not too far off from the original norm. The approximation we use here comes from the following version of John's theorem.

\begin{prop} \label{prop:johnThm}
Let $q \in \N$, and let $E \subset \Bc$ be a $q$-dimensional subspace. Then, there exists an inner product $(\cdot, \cdot)_E$ on $E$ for which the following hold.
\begin{enumerate}
\item The norm $\|\cdot\|_E$ induced by $(\cdot, \cdot)_E$ satisfies the inequality
	\begin{align} \label{eq:johnThmBound}
		\frac{1}{\sqrt{q}} |v| \leq \|v\|_E \leq \sqrt{q} |v|.
	\end{align}
	for all $v \in E$.
\item The induced volume $m_E$ coincides with the Lebesgue volume on $E$ arising from the inner product $(\cdot, \cdot)_E$.
\end{enumerate}
\end{prop}
\begin{proof}
This is merely a re-telling of John's theorem (Theorem 15 in Chapter 4 of \cite{Bo}), which asserts that given any convex body $C \subset \R^n$ which contains $0$ in its interior and is symmetric about the origin, there is a unique ellipsoid $D$ containing $C$, having minimal volume among all ellipsoids containing $C$; additionally, $D$ satisfies the inclusion $\frac{1}{\sqrt{n}} D \subset C$. 

Noting that ellipsoids centered at $0$ in $\R^n$ are in one-to-one correspondence with inner products on $\R^n$, we let $(\cdot, \cdot)$ refer to the inner product from John's Theorem applied to the unit ball $B_E$ of $(E, |\cdot|)$.  We modify $(\cdot, \cdot)$  by a scalar to define the inner product $(\cdot, \cdot)_E$: the scalar is chosen so that the Lebesgue volume of $(\cdot, \cdot)_E$ and the induced volume $m_E$ coincide. By Item 2 of Lemma \ref{lem:indVol}, this is ensured by the condition $m_E \{v \in E : \|v\|_E \leq 1\} = \omega_q$. 

The bound in \eqref{eq:johnThmBound} follows from the fact that the norm $\|\cdot\|$ induced by the John's theorem inner product $(\cdot, \cdot)$ satisfies $\|v\| \leq |v| \leq \sqrt{q} \|v\|$ for all $v \in E$.
\end{proof}

\noindent {\bf Notation.} For the remainder of this subsection, $q \in \N$ is fixed, and if $a,b > 0$ are real numbers, we will write $a \lesssim b$ if there is a constant $C_q$ depending only on $q$ for which $a \leq C_q b$, and similarly for $a \gtrsim b$. We write $a \approx b$ if $a \lesssim b$ and $a \gtrsim b$ hold. For example, \eqref{eq:johnThmBound} in Proposition \ref{prop:johnThm} can be written as $ \|v\|_E \approx |v|$ for $v \in E$.

We now pursue the program of using $(\cdot, \cdot)_E, \| \cdot\|_E$ to deduce approximate identities for the induced volume $m_E$. There are two applications we will cover in this section: the first is a way of estimating `block' determinants, and the second involves estimating determinants in terms of special bases.

\smallskip
\noindent{\bf Application 1: Estimating `Block Determinants'.}

Let $V = E \oplus F$ be a splitting of a finite dimensional subspace $V \subset \Bc$, where $\dim V = q$ and $\dim E = k = q - \dim F$ for some $k < q$. For $A \in L(\Bc)$, our goal is to estimate $\det(A | V)$ in terms of the product $\det(A | E) \cdot \det(A | F)$.

The point of departure for us is the explicit formula for the volume of a parallelepiped in an inner product space: when $v_1, \cdots, v_q \in V$,
\begin{align}\label{eq:massForm}
m_V P[v_1, \cdots, v_q] = \|v_q\|_V \cdot \prod_{i = 1}^{q - 1} d_V(v_i, \langle v_j : i < j \leq q \rangle),
\end{align}
where $d_V$ refers to the minimal distance taken using the $\|\cdot\|_V$ norm. 

Let $v_1, \cdots, v_k$ and $v_{k + 1}, \cdots, v_q$ be $|\cdot|$-unit vector bases for $E, F$ respectively, and observe that because of this choice,
\[
m_V P[v_1, \cdots, v_q] = m_V P[v_{k + 1}, \cdots, v_q] \cdot \prod_{i =1 }^k d_V(v_i, \langle v_j : i < j \leq q \rangle )
\]
holds by formula \eqref{eq:massForm} for $m_V P[v_{k + 1}, \cdots, v_q]$, where we abuse notation and let $m_V $ refer also to the volume on $F$ arising from the restriction of $\|\cdot\|_V$ to $F$. On the other hand, $\|v\|_V \approx |v| \approx \|v\|_F$ by \eqref{eq:johnThmBound} for all $v \in F$, from which we deduce $m_V P[v_{k + 1}, \cdots, v_q] \approx m_F P[v_{k + 1}, \cdots, v_q]$ by Lemma \ref{lem:indVol}, Items 2 and 3. Collecting this,
\begin{align}\label{eq:splitParaEst2}
m_V P[v_1, \cdots, v_q] \approx m_F P[v_{k + 1}, \cdots, v_q] \cdot \prod_{i = 1}^k d(v_i, \langle v_j : i < j \leq q \rangle ).
\end{align}
Observe that $d(v_i, \langle v_j : i < j \leq q \rangle) = \sin \theta(\langle v_i \rangle, \langle v_j : i < j \leq q \rangle) = | \pi_{\langle v_i \rangle \ds \langle v_j : j < i \leq q \rangle}|^{-1}$ by Lemma \ref{lem:angleForm}, so we must estimate $| \pi_{\langle v_i \rangle \ds \langle v_j : j < i \leq q \rangle}|$ from above and below. Note that $\pi_{\langle v_i \rangle \ds \langle v_j : j < i \leq q \rangle} = \pi_{\langle v_i \rangle \ds \langle v_j : i < j \leq k \rangle} \circ \pi_{E \ds F}$ and $\pi_{\langle v_i \rangle \ds \langle v_j : i < j \leq k \rangle} = \pi_{\langle v_i \rangle \ds \langle v_j : i < j \leq q \rangle} |_E$, so it follows that
\begin{align*}
d(v_i, \langle v_j : i < j \leq k \rangle )^{-1} & =  |\pi_{\langle v_i \rangle \ds \langle v_j : i < j \leq k \rangle}|  \leq | \pi_{\langle v_i \rangle \ds \langle v_j : j < i  \leq q \rangle}| \\ & \leq |\pi_{E \ds F}| \cdot |\pi_{\langle v_i \rangle \ds \langle v_j : i < j \leq k \rangle}| 
= |\pi_{E \ds F}| \cdot d(v_i, \langle v_j : i < j \leq k \rangle )^{-1}.
\end{align*}
An application of \eqref{eq:massForm} to $m_E P[v_1, \cdots, v_k]$ now implies
\begin{align} \label{eq:paraMassSplit}
|\pi_{E \ds F}|^{-1} \lesssim \frac{m_V P[v_1, \cdots, v_q]}{m_E P[v_1, \cdots, v_k] \cdot m_F P[v_{k + 1}, \cdots, v_q]} \lesssim 1,
\end{align}
as long as the denominator of the central term in \eqref{eq:paraMassSplit} is nonzero.

Having estimated in this way, one can now estimate the determinant $\det(A | V)$ by considering \eqref{eq:paraMassSplit} applied to the parallelepipeds $P[v_1, \cdots, v_q]$ and $P[A v_1, \cdots, A v_q] = A P[v_1, \cdots, v_q]$. We collect the results in the following lemma. 

\begin{lem} \label{lem:detSplit}
Let $A \in L(\Bc)$ and $V \subset \Bc$ have dimension $q$. Let $V = E \oplus F$ be a splitting with $\dim E = k < q$, and assume that $A|_V$ is injective. Writing $V' = A V, E' = A E, F' = A F$, we have the estimate
\[
\frac{1}{|\pi_{E' \ds F'}|^k} \lesssim \frac{\det(A | V)}{\det(A | E) \det(A | F)} \lesssim |\pi_{E \ds F}|^k.
\]
\end{lem}

\begin{rmk} \label{rmk:sampleVolProof}
Observe that Lemma \ref{lem:sampleVol} for a $q$-dimensional subspace $E \subset \Bc$ follows straightaway from \eqref{eq:massForm} applied to the norm $\|\cdot\|_E$, using \eqref{eq:johnThmBound} to compare $|\cdot|$ with $\|\cdot\|_E$.
\end{rmk}

\smallskip
\noindent{ \bf Application 2: Estimating Determinants from Bases.}

Let $q \in \N$, $A \in L(\Bc)$, and let $V \subset \Bc$ be a $q$-dimensional subspace. We will give here an estimate for $\det(A | V)$ in terms of a special kind of basis for $V$.  To begin, form a $(\cdot, \cdot)_V$-orthonormal set $v_1, \cdots, v_q$ of $V$ and recall the Hadamard bound for the determinant:
\begin{align}\label{eq:hadamardBoundIP}
\det(A | V) \leq \prod_{i = 1}^q \|A v_i \|_{V'},
\end{align}
where $V' = A V$. 

More refined information is realized using the Singular Value Decomposition (SVD): if $\{v_i\}$ is an orthonormal SVD basis for $A|_V$ on $(V, (\cdot, \cdot)_V)$ (that is, a $(\cdot, \cdot)_V$-orthonormal basis of eigenvectors of $(A|_V)^* A|_V$, where the adjoint $(A|_V)^*$ is taken with respect to $(V, (\cdot, \cdot)_V)$ and $(V', (\cdot, \cdot)_{V'})$), then 
\begin{align}\label{eq:SVDformula}
\det(A | V) = \prod_{i = 1}^q \|A v_i\|_{V'},
\end{align}
and when $A|_V$ is injective, the set $\{w_i\}$ defined by $w_i = A v_i / \|A v_i\|_{V'}$ forms a $(\cdot, \cdot)_{V'}$-orthonormal set (as $\{w_i\}$ is a set of eigenvectors for $(A|_V) (A|_V)^*$). 

All this motivates the following definition.
\begin{defn}
Let $V \subset \Bc$ be a subspace of dimension $q$. A basis of $|\cdot|$-unit vectors $v_1, \cdots, v_q$ for $V$ is called an \emph{almost orthonormal}
basis if it is $(\cdot, \cdot)_V$-orthogonal.
\end{defn}

We now apply \eqref{eq:johnThmBound} to each of \eqref{eq:hadamardBoundIP} and $\eqref{eq:SVDformula}$, obtaining the following.

\begin{prop}\label{prop:approxSVD}
Let $A  \in L(\Bc)$ and let $V \subset \Bc$ be a subspace with $\dim V = q$.  

\begin{enumerate}
\item (Hadamard Bound) If $v_1, \cdots, v_q$ is an almost orthonormal basis of $V$, then
\[
\det(A|V) \lesssim  \prod_{i = 1}^q |A v_i|
\]

\item (Singular Value Decomposition) There exists an almost orthonormal basis $v_1, \cdots, v_q$ for $V$ such that
\[
 \det(A) \approx \prod_{i = 1}^q |A v_i| 
\]
Write $V' = A V$. If $A : V \to V'$ is invertible, then the basis $w_i = A v_i / |A v_i|$ for $V'$ is almost orthonormal.
\end{enumerate}
\end{prop}

The following is a simple corollary of Proposition \ref{prop:approxSVD}, relating the minimal expansion of a linear map on a finite dimensional subspace to volume growth on that subspace. 

\begin{cor} \label{cor:minExpand}
Let $A \in L(\Bc)$ and $E \subset \Bc$ be a $q$-dimensional subspace. For any $v \in E \setminus \{0\}$,
\[
\frac{|Av|}{|v|} \cdot |A|_E|^{q-1} \gtrsim \det(A|E).
\]
\end{cor}


\subsubsection{The Determinant and Measurability}

We now give a lemma on the continuity properties of the determinant and the maximal $q$-dimensional growth $V_q$, defined for $A \in L(\Bc)$ by
\[
V_q(A) = \sup\{\det(A | E) : \dim E = q\}.
\]

\begin{lem}\label{lem:measDet}
Let $q \in \N$, and let $E \subset \Bc$ be $q$-dimensional subspace.

\begin{enumerate}
\item If $T_n \to T$ in norm on $L(\Bc)$, then $\det(T_n | E) \to \det(T | E)$.

\item If $T_n \to T$ in norm on $L(\Bc)$, then $ V_q(T_n) \to V_q(T)$.

\end{enumerate}
\end{lem}


\begin{proof}
We begin by proving the following.

\begin{cla} \label{cla:detContSingleSub}
If $A_n, A : V \to V'$ are injective linear maps between $q$-dimensional normed vector spaces $V, V'$, and $A_n \to A$ in operator norm, then $\det(A_n | V) \to \det(A | V)$.
\end{cla}
\begin{proof}[Proof of Claim]
For two injective maps $A, B : V \to V'$, it is simple to show that
\[
B^{-1} A \big( B_V \big) \subset (1 + |B^{-1}| \cdot  |A - B|) B_V,
\]
where $B_V$ is the closed unit ball of $V$. Using Item 3 of Lemma \ref{lem:indVol}, it follows that $m_V \big( B^{-1} A (B_V) \big) \leq (1 + |B^{-1}| \cdot  |A - B|)^q$, and since $\det(B^{-1} A | V) = \omega_q^{-1} \cdot m_{V} \big( B^{-1} A (B_V) \big)$, it follows from the multiplicativity of the determinant as in Item 3 of Proposition \ref{prop:detProps} that
\[
\frac{\det(A | V)}{\det(B | V)} \leq (1 + |B^{-1}| \cdot |A - B|)^q.
\]
Exchanging the roles of $A$ and $B$ and taking a logarithm, we arrive at
\begin{equation}\label{eq:detContEst}
\bigg| \log  \frac{\det(A | V)}{\det(B | V)} \bigg| \leq q \log \bigg( 1 + \big( |A^{-1}| \vee |B^{-1}| \big) \cdot |A - B| \bigg),
\end{equation}
where $a \vee b := \max \{a, b\}$ for $a,b \in \R$. We will also make use of the elementary estimate
\begin{equation}\label{eq:invNormComp}
|B^{-1}| \leq \big( |A^{-1}|^{-1} - |A - B| \big)^{-1},
\end{equation}
which holds when the parenthetical quantity on the RHS of \eqref{eq:invNormComp} is nonnegative.

Applying \eqref{eq:detContEst} to the situation $B = A_n$, and taking $n$ sufficiently large so that $|A_n^{-1}| \leq 2 |A^{-1}|$, by \eqref{eq:invNormComp}, we see now that $|A_n - A| \to 0$ implies $|\log \det(A_n^{-1} \circ A | V)| \to 0$, i.e., $\det(A_n | V) \to \det(A | V)$, as desired.
\end{proof}

\noindent{\bf Proof of Item 1.} If $\det(T | E) = 0$, let $v \in E$ be a unit vector for which $T v = 0$. By Corollary \ref{cor:minExpand},
\[
\det(T_n | E) \leq C_q |T_n|_E|^{q-1} |T_n v|,
\]
where $C_q > 0$ depends on $q$ alone. The right hand side goes to zero as $n \to \infty$ when $T_n|_E \to T|_E$ in norm, and so $\det(T_n | E) \to 0$. So, from here on we assume $\det(T|E) \neq 0$.

To complete the proof of Item 1 in this case, one can compute using Item 1 of Lemma \ref{lem:apertureProps} that
\begin{align} \label{eq:imageSubComp}
d_H(T_n E, T E) \leq 2 \big( |(T|_E)^{-1}| \vee |(T_n|_E)^{-1}| \big) \cdot |T_n - T|.
\end{align}
By \eqref{eq:invNormComp}, for $n$ large enough, $|(T_n|_E)^{-1}| \leq 2 |(T|_E)^{-1}|$, hence $d_H(T_n E, T E) \to 0$ as $n \to \infty$.

It now follows by Proposition \ref{prop:openCond} that if $F \subset \Bc$ is a fixed topological complement to $T E$, then $T_n E$ is complemented to $F$ for $n$ sufficiently large. Since $\pi_{T_n E \ds F}|_{T E} \circ \pi_{T E \ds F}|_{T_n E} = \Id_{T_n E}$, the identity on $T_n E$, we may now decompose
\begin{align}\label{eq:detDecomp}
\frac{\det(T_n |E )}{\det(T | E)} = \det(\pi_{T_n E \ds F} | T E)\cdot  \bigg( \frac{ \det(\pi_{T E \ds F} \circ T_n | E) }{\det(T | E)} \bigg).
\end{align}
Note that the parenthetical term on the RHS of \eqref{eq:detDecomp} goes to $1$ as $n \to \infty$ by Claim \ref{cla:detContSingleSub}, because $T|_E = \pi_{T E \ds F} \circ T|_E$, so that $\pi_{T E \ds F} \circ T_n|_E \to T|_E$ in norm in $L(E, TE)$. So, we estimate the remaining term on the RHS of \eqref{eq:detDecomp}; however, since $T_n E \to T E$ in the Hausdorff distance, it follows by \eqref{eq:graphNormEst} that $\pi_{F \ds T_n E}|_{T E} \to 0$ in norm as $n \to \infty$. Because $\pi_{T_n E \ds F}|_{T E} = \Id_{T E} - \pi_{F \ds T_n E}|_{T E}$, it follows by a simple argument involving Lemma \ref{lem:indVol} that $\det(\pi_{T_n E \ds F} | T E) \to 1$, as desired.

\smallskip
{\bf Proof of Item 2.} If $V_q(T) = 0$, then an argument similar to the beginning of the proof of Item 1 lets us conclude $V_q(T_n) \to V_q(T)$. Hereafter we assume that $V_q(T) > 0$.

Note that 
\begin{align}\label{eq:VqlowerBound}
V_q(T) \leq \liminf_{n \to \infty} V_q(T_n)
\end{align}
follows immediately from Item 1: for any subspace $E \subset \Bc$ with dimension $q$, $\det(T | E) = \lim_{n \to \infty} \det(T_n | E) \leq \liminf_{n \to \infty} V_q(T_n)$. So, it suffices to show $\limsup_{n \to \infty} V_q(T_n) \leq V_q(T)$.

For each $n$, let $E_n \subset \Bc$ be a subspace of dimension $q$ for which $V_q(T_n) \leq (1 + \frac{1}{n} ) \det(T_n | E_n)$. We take $n$ large enough so that $V_q(T) \leq 2 V_q(T_n)$, using \eqref{eq:VqlowerBound}. Now, by Corollary \ref{cor:minExpand},
\[
V_q(T) \leq 2 V_q(T_n) \leq 2 \bigg(1 + \frac{1}{n} \bigg) \det(T_n | E_n) \leq 4 C_q |T_n|^{q-1} \cdot |(T_n |_{E_n})^{-1}|^{-1},
\]
where $C_q > 0$ depends on $q$ alone, and so taking $n$ large enough so that $|T_n| \leq 2 |T|$, we have that $|(T_n|_{E_n})^{-1}| \leq D$, where $D$ is a constant independent of $n$. Applying \eqref{eq:invNormComp}, we also obtain $|(T|_{E_n})^{-1}| \leq 2 D$ for $n$ sufficiently large.

Because we can bound $|(T_n|_{E_n})^{-1}|, |(T|_{E_n})^{-1}|$ from above independently of $n$, we can apply \eqref{eq:imageSubComp} to the subspace $E_n$ in place of $E$, obtaining $d_H(T_n E_n, T E_n) \to 0$ as $n \to \infty$. For each $n$, let $F_n$ be a complement to $E_n$ with $|\pi_{T E_n \ds F_n} | \leq \sqrt{q}$ (Lemma \ref{lem:compExist}). So, by Proposition \ref{prop:openCond}, we deduce that for $n$ sufficiently large, $T_n E_n$ complements $F_n$. We now decompose
\begin{align}\label{eq:detDecomp2}
\frac{\det(T_n | E_n)}{ \det(T | E_n)} = \det(\pi_{T_n E_n \ds F_n} | T E_n) \cdot \bigg( \frac{\det(\pi_{T E_n \ds F_n} \circ T_n |E_n)}{\det(T | E_n)}\bigg),
\end{align}
analogously to \eqref{eq:detDecomp}. The estimate \eqref{eq:detContEst} and the fact that $|(T_n|_{E_n})^{-1}|, |(T|_{E_n})^{-1}| \leq 2D$ imply that the parenthetical term on the RHS of \eqref{eq:detDecomp2} goes to $1$ as $n \to \infty$. The fact that $\pi_{F_n \ds T_n E_n}|_{T E_n} \to 0$ in norm (following from \eqref{eq:graphNormEst} and that $d_H(T_n E_n, T E_n) \to 0$) implies that the remaining RHS term of \eqref{eq:detDecomp2} goes to 1, analogously to the proof of Item 1. Therefore, the LHS of \eqref{eq:detDecomp2} goes to 1 as $n \to \infty$, and we conclude that $\limsup_{n \to \infty} V_q(T_n) \leq  V_q(T) $, as desired.
\end{proof}

\subsubsection{Global Geometry and the Induced Volume}

We will now discuss properties of the induced volume which depend on the Banach space at large. We will give a relation between $V_q$ and the Gelfand numbers, and a relation to the measure of noncompactness $|\cdot|_{\a}$.

\smallskip
\noindent{\bf The Gelfand Numbers and Maximal Volume Growth.}

For an operator $A$ on a Hilbert space, $V_q(A)$ (defined by \eqref{eq:maxVolGrow}) is equal to the product of the first $q$ singular values of $A$, i.e., the product of the highest $q$ eigenvalues of $\sqrt{A^* A}$ (Proposition 1.4 in Chapter V of \cite{Temam}).  As a consequence, one can think of the $q$-th Lyapunov exponent (with multiplicity) of a cocycle as the generic asymptotic growth rate of the $q$-th singular value of high iterates of the cocycle \cite{Ra, R1}. In this section, we will recover this identification for a suitable generalization of `singular value' to the Banach space setting, a fact to be put to use in Section 3.

The idea of `singular values' has been generalized in several nonequivalent ways for Banach spaces: for an account, see \cite{P}. The definition we shall employ here is that of the \emph{Gelfand numbers} $c_q(\cdot)$, $q \in \N$, defined for $A \in L(\Bc)$ by
\begin{align} \label{eq:gelfand}
c_q(A) = \inf\{|A|_R| : R \text{ closed, } \codim R = q-1\},
\end{align}
with $c_1(A) := |A|$ (Chapter 2, Section 4 of \cite{P}). Notice that $\{c_q(A)\}_q$ is a decreasing sequence, and so we define $c(A) := \lim_m c_m(A) = \inf_m c_m(A)$. Observe that $c(A) = \inf \{|A|_R| : R \text{ closed, } \codim R < \infty\}$; this is a different kind of measure of noncompactness from $|\cdot|_{\a}$ (2.4.10 in \cite{Akhmerov}), and as we will see in Lemma \ref{lem:gelfandNoncompact}, is equivalent to $|\cdot|_{\a}$.

We will now show that $V_q(A)$ is approximated by the product of the first $q$ Gelfand numbers $c_i(A)$.

\begin{lem}\label{lem:gelfandProps}
Let $q \in \N$. For any $A \in L(\Bc)$,
\begin{align}\label{eq:prodGelfand}
V_q(A) \approx c_{q}(A) V_{q-1}(A) .
\end{align}
where $\approx$ depends on $q$ alone.
\end{lem}
\begin{proof}
We prove $\lesssim, \gtrsim$ separately. Assume $V_{q-1}(A) > 0$, since otherwise, $A$ is a finite rank operator of rank $\leq q - 2$ and \eqref{eq:prodGelfand} is trivial.

To prove $\gtrsim$, let $V \subset \Bc$ be a subspace of dimension $q$ for which $\det(A | V) \geq \frac{1}{2} V_q(A)$, and let $R \subset \Bc$ have codimension $q-1$ such that $|A|_R| \leq 2 c_q(A)$.

 Note that $\dim R \cap V \geq 1$ by the definition of codimension; fix $w \in R \cap V$ with $|w| = 1$. Let $V_0 \subset V$ be a complement to $\langle w \rangle$ in $V$ for which $|\pi_{\langle w \rangle \ds V}| = 1$ (by Lemma \ref{lem:compExist}). We now estimate using the `block determinant' estimate in Lemma \ref{lem:detSplit}:
\[
V_q(A) \leq 2 \det(A | V) \lesssim \det(A | V_0) \det(A | \langle w \rangle) \leq V_{q - 1}(A) |A w| \leq 2 V_{q - 1}(A) c_q(A),
\]
which is what we wanted.

To prove $\lesssim$, let $E \subset \Bc$ be a subspace of dimension $q-1$ for which $\det(A | E) \geq \frac{1}{2} V_{q-1}(A)$.  Now, $A E$ has dimension $q - 1$ (from $V_{q - 1}(A) > 0$), so by Lemma \ref{lem:compExist}, we can find a $(q-1)$-codimensional complement $F_2 \subset \Bc$ for which $|\pi_{A E \ds F_2}| \leq \sqrt{q-1}$. By Lemma \ref{lem:splittingExist}, $F_1 := \{v \in \Bc : A v \in F_2\}$ has codimension $q-1$.

Let $w \in F_1, |w| = 1$. Using the lower bound in Lemma \ref{lem:detSplit}, we estimate
\[
V_q(A) \geq \det(A | E \oplus \langle w \rangle) \gtrsim \det(A | E) \det(A | \langle w \rangle) \geq \frac{1}{2} V_{q -1 }(A) |A w|
\]
This inequality holds for any $w \in F_1, |w| = 1$, and so implies an upper bound on $|A|_{F_1}|$, which is $\geq c_q(A)$ by definition. This completes the estimate.
\end{proof}

\smallskip
\noindent{\bf Measures of Noncompactness and Maximal Volume Growth.}

Like singular values of operators in Hilbert space, the Gelfand numbers can detect compactness and estimate the measure of noncompactness $|\cdot|_{\a}$ in the following sense. 

\begin{lem}[2.5.5 in \cite{Akhmerov}]\label{lem:gelfandNoncompact}
Let $A \in L(\Bc)$. Then,
\[
\frac{1}{2} |A|_{\a} \leq c(A) \leq 2 |A|_{\a}.
\]
\end{lem}

The maximal volume growths $V_q$ also have a relation to $|\cdot|_{\a}$. 
\begin{prop} \label{lem:maxVolNoncompact}
Let $A \in L(\Bc)$. Using the convention $\log 0 =  - \infty$, we have that
\begin{align}\label{eq:VQlimit}
\limsup_{q \to \infty} \frac{1}{q} \log V_q(A) \leq \log ( 2 |A|_{\a}).
\end{align}
\end{prop}
\begin{proof}
For $x \in \Bc$ and $r > 0$, we denote $B(x, r) = \{v \in \Bc : |v - x| < r\}$.

If $A$ is an operator of finite rank, then \eqref{eq:VQlimit} holds vacuously, and so we can assume without loss that $V_q(A) > 0$ for all $q \in \N$.

Fix $r > |A|_{\a}$ and let $\{B(x_i, r)\}_{i = 1}^{C_r}$ be a finite cover of $A\big( B(0,1)\big)$ by balls of radius $r$ centered at points $x_i \in \Bc$. For each $q \in \N$, let $E_q \subset \Bc$ be a $q$-dimensional subspace for which $V_q(A) \leq 2 \det(A | E_q)$. Writing $A E_q = E'_q$, note that $\dim E'_q = q$, and that
\[
m_{E'_q} \bigg( E'_q \cap A \big( B(0,1) \big) \bigg) = m_{E'_q} \big( A (B_{E_q}) \big) \leq  \sum_{i = 1}^{C_r} m_{E'_q} \big( E'_q \cap B(x_i, r) \big).
\]
As one can easily check, for Hilbert spaces we have that $m_{E'_q} \big( E'_q \cap B(x_i, r) \big) \leq  r^q \omega_q$, with equality when $x_i \in E'_q$. For Banach spaces, the following can be recovered.
\begin{cla}\label{cla:ballVolEst}
Let $E \subset \Bc$ be a $q$-dimensional subspace with $\dim E = q < \infty$. For any $x \in \Bc$, we have that $m_E \big( E \cap B(x, r) \big) \leq (2 r)^q \omega_q$.
\end{cla}
From Claim \ref{cla:ballVolEst}, we obtain
\[
m_{E_q'} \big( E_q' \cap A B(0,1)\big) \leq C_r (2 r)^q \omega_q,
\]
hence $V_q(A) \leq 2 \det(A | E_q) \leq  2C_r (2 r)^q$. This implies \eqref{eq:VQlimit}.
\end{proof}

\begin{proof}[Proof of Claim \ref{cla:ballVolEst}]
Let $F$ be a topological complement to $E$ in $\Bc$ (Lemma \ref{lem:compExist}), and decompose $x = e + f$. If $f = 0$, then $x \in E$ and Claim \ref{cla:ballVolEst} is obvious: so, we may assume $f \neq 0$. Moreover, without loss, we can take $e = 0$ by the the translation invariance of $m_E$ on $E$.

Applying the Riesz Lemma (Lemma 4.7 of \S 4.1 in \cite{S}) to $E \subset E \oplus \langle f \rangle$, there exists a unit vector $g \in E \oplus \langle f \rangle$ for which $d(g, E) = 1$; equivalently, $|\pi_{\langle g \rangle \ds E}| = 1$ (Lemma \ref{lem:angleForm}).

Writing $f = e_0 + a g$, with $a \in \R$, $e_0 \in E$, we will show that $B(f, r) \cap E \subset B(e_0, 2 r)$, which implies Claim \ref{cla:ballVolEst}. To see this, we estimate: if $v \in B( f, r) \cap E$, then $|v - f| < r$, and so
\[
|v - e_0| = |\pi_{E \ds \langle g \rangle} (v - (e_0 + a g))| \leq |\pi_{E \ds \langle g \rangle}| \cdot  |v - f| < 2 r,
\]
which is what we wanted.
\end{proof}

\section{Lyapunov Exponents for Banach Space Cocycles}

In this section, we prove our main result, Theorem \ref{thm:volMET}, by emulating Ruelle's proof in \cite{R1,R2}. We will begin by using the measurability assumption \eqref{eq:measCondition} and Lemma \ref{lem:measDet} to obtain the Lyapunov exponents and show they may accumulate only at the asymptotic exponential growth rate $l_{\a}$ of $|T^n_x|_{\a}$. Then, we will state our primary tool, Proposition \ref{prop:staticMET}, which should be thought of as a `trajectory-wise' version of the MET which extracts the `slow' growing subspace corresponding to the second Lyapunov exponent, and show using an induction procedure (Lemma \ref{lem:lyapInduct}) how to complete the proof. Remaining at that point will be to prove Proposition \ref{prop:staticMET} and Lemma \ref{lem:lyapInduct}, and to prove the volume growth \eqref{eq:correctVolGrow} in Theorem \ref{thm:volMET}, which we formulate as Lemma \ref{lem:rightVolGrow}.

\subsection{Lyapunov Exponents for $T$}

In this section, we find the Lyapunov exponents for the cocycle $T$ using growth rates, and prove their basic properties. To begin, the following is an immediate consequence of Item 2 in Lemma \ref{lem:measDet}.
\begin{cor} \label{cor:volGrowMeas}
Assume that the cocycle $T : X \to L(\Bc)$ satisfies the measurability hypothesis \eqref{eq:measCondition}. Then, for any $n, q \geq 1$, the map $x \mapsto V_q(T^n_x)$ is measurable as a map $(X, \Fc) \to (\R, \operatorname{Bor}(\R))$.
\end{cor}

\noindent We note as well that the map $x \mapsto |T^n_x|_\a$ is measurable for $n \geq 1$ (note $|A|_\a \leq |A|$ for $A \in L(\Bc)$). The following identifies the Lyapunov exponents of the cocycle $T$ in terms of volume growth.

\begin{lem}\label{lem:growthRate}
Let $T : X \to L(\Bc)$ be a map satisfying \eqref{eq:measCondition}. Assume as well that  
\begin{align}\label{eq:cocycleIntegrable}
\int_X \log^+ |T_x| \, d \mu(x) < \infty.
\end{align}

Then, the following hold.
\begin{enumerate}
\item For any $q \geq 1$, the exponential growth rates $l_q$, defined by
\begin{align} \label{eq:definitionLQ}
l_q = \lim_{n \to \infty} \frac{1}{n} \log V_q(T^n_x),
\end{align}
exist and are constant $\mu$-almost surely. 
\item Writing $K_1 = l_1, K_q = l_q - l_{q - 1}$ for $q > 1$, the sequence $\{K_q\}_{q \geq 1}$ is nonincreasing, i.e., $K_1 \geq K_2 \geq \cdots$.
\item Defining $l_{\a} = \lim_{q \to \infty} K_q $, we have that $\mu$-almost surely,
\[
 \lim_{n \to \infty} \frac{1}{n} \log |T^n_x|_{\a} = l_{\a}
\]
\end{enumerate}
\end{lem}

\begin{rmk}\label{rmk:exponents}
We write $\l_1 > \l_2 > \cdots$ for the distinct values of the sequence $\{K_q\}_q$. There may be finitely many of these, in which case the last of these values is equal to $l_{\a}$, or infinitely many, in which case $\l_i \to l_{\a}$ as $i \to \infty$. We write $m_i$ for the multiplicity of the value $\l_i$ amongst the sequence $\{K_q\}_q$ and $M_1 := 0, M_i := m_1 + \cdots + m_{i - 1}$ for $i \geq 2$.
\end{rmk}

\begin{proof}
{\bf Item 1.} The almost sure convergence of $\frac{1}{n} \log V_q(T^n_x)$ follows immediately from the Kingman Subadditive Ergodic Theorem (KSET) (\S 1.5 of \cite{Kr}), in light of the integrability hypothesis \eqref{eq:cocycleIntegrable} and the measurability of $x \mapsto V_q(T^n_x)$ (Corollary \ref{cor:volGrowMeas}), since
\[
V_q(T^{n + m}_x) \leq V_q(T^m_{f^n x}) \cdot V_q(T^n_x) 
\] 
by the multiplicativity of the determinant as in Proposition \ref{prop:detProps}. 

{\bf Item 2.} From Lemma \ref{lem:gelfandProps}, one has that
\[V_q(T^n_x) \approx c_1(T^n_x) \cdot c_2(T^n_x) \cdots c_q(T^n_x),\]
where $\approx$ depends only on $q$, and so one can see directly from the convergence of the sequences $\frac{1}{n} \log V_q(T^n_x)$ as $n \to \infty$ that the growth rates of the Gelfand numbers $\frac{1}{n} \log c_q(T^n_x)$ converge to $K_q$ as $n \to \infty$ for any $q \geq 1$. Obviously, $c_{q}(\cdot) \geq c_{q + 1}(\cdot)$ for any $q \geq 1$, and so $K_1 \geq K_2 \geq \cdots$ follows immediately.

{\bf Item 3.} Recall that $c(A) = \inf_q c_q(A)$. Note that almost surely,
\[
K_q = \lim_{n \to \infty} \frac{1}{n} \log c_q(T^n_x) \geq \limsup_{n \to \infty} \frac{1}{n} \log c(T^n_x).
\]
By Lemma \ref{lem:gelfandNoncompact}, taking $q \to \infty$ lets us deduce $l_{\a} \geq \lim_{n \to \infty} \frac{1}{n} \log |T^n_x|_{\a}$.

For the other direction, the limit as $q \to \infty$ of the Cesaro averages $\frac{1}{q} l_q$ equals $l_{\a}$. Recall that $\frac{1}{q} l_q$ is decreasing, and that by the KSET, the limit in \eqref{eq:definitionLQ} holds almost surely and in $L^1$ (\S 1.5 of \cite{Kr}); in particular, $l_q = \inf_n \frac1n \int \log V_q(T^n_x) \, d \mu(x)$. Now, using Proposition \ref{lem:maxVolNoncompact},
\begin{align*}
l_{\a} &= \inf_q \frac{1}{q} l_q = \inf_q \inf_n \frac{1}{q n} \int \log V_q(T^n_x) \, d \mu(x) = \inf_n \inf_q \frac{1}{qn } \int \log V_q(T^n_x) \, d \mu(x)\\
& \leq \inf_n \frac{1}{n} \limsup_q  \int \frac{1}{q} \log V_q(T^n_x) \, d \mu(x) \leq \inf_n \frac{1}{n} \int \log (2 |T^n_x|_{\a}) \, d \mu(x) \leq \lim_{n \to \infty} \frac{1}{n} \int \log |T^n_x|_{\a} d \mu(x) \, .
\end{align*}
In the second line we use the Reverse Fatou Lemma. By the KSET we have that $\lim_n \frac{1}{n}  \log |T^n_x|_{\a}$ exists and coincides almost surely with $\lim_n \frac{1}{n} \int \log |T^n_x|_{\a} d \mu(x)$. This completes the estimate.
\end{proof}

\subsection{Proof of Theorem \ref{thm:volMET}}

For the remainder of this section, we give the volume-based proof of Theorem \ref{thm:volMET}. 

\subsubsection{A `trajectory-wise' version of the MET}

Below, we state a version of the MET, to be applied one trajectory at a time: this the analogue of Proposition 2.1 in Section 2 of Ruelle's paper \cite{R2}. We refer to the norms on the spaces $V_i$ below with the same symbol $|\cdot|$.

\begin{prop} \label{prop:staticMET}
Let $V_0, V_1, V_2, \cdots$ be Banach spaces and let $T_i : V_i \to V_{i + 1}$ be a sequence of bounded linear maps. Write $T^n = T_{n-1} \circ \cdots \circ T_0$, and assume the following of $\{T_n\}$.

\begin{enumerate}
\item $\lim_{n \to \infty} \frac{1}{n} \log^+ |T_n| = 0$.

\item For any $q \in \N$, the following limits exist:
\[
L_q = \lim_{n \to \infty} \frac{1}{n} \log V_q(T^n).
\]
\item 
Writing $k_1 = L_1, k_q = L_q - L_{q-1}$ for $q > 1$, we have that for some $m < \infty$, $\ol := k_1 = \cdots = k_m > k_{m + 1} =: \ul$. 
\end{enumerate}
Then, the subspace $\underline{F} := \{v \in V_0 : \limsup_{n \to \infty} \frac{1}{n} \log |T^n v| \leq \ul \}$ is closed and $m$-codimensional. For any $v \in V_0 \setminus \underline{F}$, 
\[
\lim_{n \to \infty} \frac{1}{n} \log |T^n v| = \ol
\]

For any $\eta \in (0,1)$, the former convergence occurs uniformly over vectors in the cone $C_{\eta}= \{v \in \Bc : d(v, \underline{F}) \geq \eta  |v| \}$ in the following sense:
\begin{align}\label{eq:uniformGrowth}
\lim_{n \to \infty} \frac{1}{n} \log \inf_{v \in C_{\eta} \setminus \{0\}}\frac{ |T^n v|}{|v|} = \ol.
\end{align}
In particular,
\begin{align}\label{eq:staticVolumeGrowth}
\lim_{n \to \infty} \frac{1}{n} \log \det(T^n | E) = m \ol
\end{align}
for any complement $E$ to $\underline{F}$.
\end{prop}

We now show how to derive all of Theorem \ref{thm:volMET} from Proposition \ref{prop:staticMET}, except for \eqref{eq:correctVolGrow}. Below, we assume that there are infinitely many distinct Lyapunov exponents- the proof for when there are finitely many distinct exponents is virtually identical.

\begin{proof}[Proof of Theorem \ref{thm:volMET} from Proposition \ref{prop:staticMET}]
We define $\Gamma \subset X$ to be the set of all $x \in X$ such that the limit $\lim_{n \to \infty} \frac{1}{n} \log V_q(T^n_x)$ exists and equals $l_q$, and for which 
\begin{equation} \label{eq:normControl}
\lim_{n \to \infty} \frac{1}{n} \log^+ |T_{f^n x}| = 0.
\end{equation}
The condition \eqref{eq:normControl} is $\mu$-generic by the Birkhoff Ergodic Theorem and the integrability hypothesis $\int_X \log^+ |T_x| d \mu(x) < \infty$.

For each $x \in \Gamma$, the sequence $T_n := T_{f^n x}$ satisfies the hypotheses of Proposition \ref{prop:staticMET}, and so we obtain the following, which can be thought of as the MET for the first Lyapunov exponent. Below, with $F_{\l}(x)$ as in \eqref{eq:slowGrowSpaces}, we define $F_i(x) := F_{\l_i}(x)$ for Lyapunov exponents $\l_i$ as in Remark \ref{rmk:exponents}.

\begin{lem}\label{lem:firstSubMET}
Let $T : X \to L(\Bc)$ satisfy the hypotheses of Lemma \ref{lem:growthRate}, and let $\l_1, \l_2, m_1$ be as in Remark \ref{rmk:exponents}. Then, for each $x \in \Gamma$, $F_2(x)$ is $m_1$-codimensional, and satisfies
\begin{gather*}
\forall v \in \Bc \setminus F_2(x),~  \lim_{n \to \infty} \frac{1}{n} \log |T^n_x v| = \l_1
\end{gather*}
The former convergence occurs uniformly over vectors in the cone $C_{\eta}(x)= \{v \in \Bc : d(v, F_2(x)) \geq \eta |v| \}$, in the following sense.
\[
\lim_{n \to \infty} \frac{1}{n} \log \min_{v \in C_{\eta}(x) \setminus \{0\}} \frac{|T^n_x v|}{|v|} = \l_1.
\]
As a consequence, for any complement $E$ to $F_2(x)$, we have that
\begin{align}\label{eq:firstSubVolGrow2}
\lim_{n \to \infty} \frac{1}{n} \log \det(T^n_x | E) = m_1 \l_1.
\end{align}
\end{lem}

We now formulate an induction step, to show that the hypotheses of Proposition \ref{prop:staticMET} are satisfied for the sequence $T_0 := T_x|_{F_2(x)}, T_n = T_{f^{n} x}$, with the Lyapunov exponents `shifted down' so as to eliminate the top exponent $\l_1$.

\begin{lem}[Exponent Extraction Lemma] \label{lem:lyapInduct}
For any $x \in \Gamma$, we have that $\lim_{n \to \infty} \frac{1}{n} \log V_q(T^n_x|_{F_2(x)})$ exists and equals $l_{q + m_1} - l_{m_1}$ for any $q \geq 1$.
\end{lem}
We will prove this in Section 3.2.3. One now has that the sequence of operators defined by $T_0 := T_x|_{F_2(x)}$, $T_n := T_{f^n x}, n > 0$, satisfies the hypotheses of Proposition \ref{prop:staticMET} with $\ol = \l_2, \ul = \l_3, m = m_2$, and the subspace $\underline{F} \subset F_2(x)$ as in the statement of Proposition \ref{prop:staticMET} is $F_3(x)$. We obtain that the codimension of $F_3(x)$ in $F_2(x)$ is $m_2$, and so $\codim F_3(x) = M_3 = m_1 + m_2$, and
\begin{gather*}
\forall v \in F_2(x) \setminus F_3(x), ~~ \lim_{n \to \infty} \frac{1}{n} \log |T^n_x v| = \l_2.
\end{gather*}

Inductively, assume that we have already shown that $F_i(x)$ has codimension $M_i$, and that for any $1 \leq j \leq i - 1$,
\begin{gather*}
\forall v \in F_j(x) \setminus F_{j + 1}(x), ~~ \lim_{n \to \infty} \frac{1}{n} \log |T^n_x v| = \l_j
\end{gather*}
for all $x \in \Gamma$.

Under this condition, one can show, repeating the argument in the proof of Lemma \ref{lem:lyapInduct}, that for any $x \in \Gamma$, the limit $\lim_{n \to \infty} \frac{1}{n} \log V_q(T^n_x|_{F_i(x)})$ exists for any $q \geq 1$ and equals $l_{q + M_i} - l_{M_i}$. Therefore, $T_0 := T_x|_{F_i(x)}, T_n := T_{f^n x}, n > 0 $ satisfies the hypotheses of Proposition \ref{prop:staticMET} with $\ol = \l_i, \ul = \l_{i  +1}, m = m_i$, and $\underline{F} = F_{i + 1}(x)$. We obtain that the codimension of $F_{i + 1}(x)$ in $F_i(x)$ is $m_i$, hence $\codim F_{i + 1}(x) = M_{i + 1}$, and that
\[
\forall v \in F_{i}(x) \setminus F_{i+ 1}(x), ~~ \lim_{n \to \infty} \frac{1}{n} \log |T^n_x v| = \l_i.
\]
This completes the induction argument.
\end{proof}

We now give the proofs of Proposition \ref{prop:staticMET} and Lemma \ref{lem:lyapInduct}, and the proof of the volume growth formula in \eqref{eq:correctVolGrow}, formulated as Lemma \ref{lem:rightVolGrow}.

\subsubsection{Proof of Proposition \ref{prop:staticMET}}

\begin{proof}[Proof of Proposition \ref{prop:staticMET}]

For each $n$, let $E_n^1$ be an $m$-dimensional subspace of $V_0$ for which \linebreak $\det(T^n | E_n^1) \geq \frac{1}{2} V_m(T^n)$, and define $E_n^2 = T^n E_n^1$. Let $F_n^2$ be a complement to $E_n^2$, as in Lemma \ref{lem:compExist}, for which $P_n^2 := \pi_{E_n^2 \ds F_n^2}$ satisfies $|P_n^2| \leq \sqrt{m}$. Let $F_n^1 = \{v \in V_0 : T^n v \in F_n^2\}$; this is a complement to $E_n^1$ by Lemma \ref{lem:splittingExist}. We write $P_n^1 = \pi_{E_n^1 \ds F_n^1}$.

To carry out our argument, we need to show that $T^n$ expands vectors on $E_n^1$ by a factor of approximately $e^{n \ol}$, that $|P_n^1|$ does not grow too quickly in $n$, and that $|T^n|_{F_n^1}|$ is bounded from above by approximately $e^{n \ul}$ on $F_n^1$. 

\smallskip
\noindent{\bf Growth on $E_n^1$.} Let $v \in E_n^1, |v| = 1$. Then, by Corollary \ref{cor:minExpand}, there is a constant $C_m$ depending only on $m$ for which
\begin{align}\label{eq:minEn1Expansion}
\frac{V_m(T^n)}{2V_1(T^n)^{m-1}} \leq \frac{\det(T^n | E_n^1)}{|T^n|^{m-1}} \leq C_m |T^n v|.
\end{align}

\smallskip
\noindent{\bf Controlling $|P_n^1|$.} Using \eqref{eq:minEn1Expansion} and the formula \eqref{eq:projectionForm} for $|P_n^1|$ in Lemma \ref{lem:splittingExist}, we have that the projection $P_n^1 := \pi_{E_n^1 \ds F_n^1}$ satisfies
\begin{align} \label{eq:Pn1Control}
|P_n^1| \leq |P_n^2| \cdot |T^n| \cdot |(T^n|_{E_n^1})^{-1}| \leq \sqrt{m} \cdot |T^n| \cdot \frac{2 C_m |T^n|^{m-1}}{V_m(T^n)} = C_m' \frac{|T^n|^m}{V_m(T^n)},
\end{align}
where $C_m'$ is again a constant depending only on $m$.

\smallskip
\noindent{\bf Bounding $|T^n|_{F_n^1}|$.} Let $v \in F_n^1, |v| =1$. We will estimate $|T^n v|$ by estimating the growth rate of the $m + 1$-dimensional subspace $E_n^1 \oplus \langle v \rangle$ under $T^n$. Treating $E_n^1 \oplus \langle v \rangle$ as a splitting and using Lemma \ref{lem:detSplit}, we have that
\[
\det(T^n|E_n^1 \oplus \langle v \rangle ) \geq C_m''^{-1} |P_n^2|_{E_n^2 \oplus \langle v \rangle}|^{-m} \cdot \det(T^n|E_n^1) \cdot \det(T^n | \langle v \rangle ),
\]
and since $\det(T^n | \langle v \rangle ) = |T^n v|$ and $\det(T^n | E_n^1 \oplus \langle v \rangle ) \leq V_{m + 1}(T^n)$, we obtain
\begin{align}\label{eq:Fn1control}
|T^n v| \leq 2 \sqrt m C_m''  \frac{V_{m+1}(T^n)}{V_m(T^n)},
\end{align}
where $C_m''$ again depends only on $m$. The RHS of \eqref{eq:Fn1control} is approximately $e^{n \ul}$ for $n$ large, as desired.

Key to this approach to the MET is showing that the subspaces $F_n^1$ converge in the Hausdorff distance at a sufficiently fast exponential rate.
\begin{cla} \label{cla:slowSpaceCauchy}
For any $\d > 0$, there exists $N \in \N$ such that for any $m \geq n  \geq N$,
\[
d_H(F_n^1, F_m^1) \leq e^{n (\ul - \ol + \d)}.
\]
\end{cla}

\begin{proof}[Proof of Claim \ref{cla:slowSpaceCauchy}]

Fix $\d > 0$.  For $v \in F_n^1, |v| = 1$, observe that $d(v, F_{n + 1}^1) \leq | P_{n + 1}^1 v|$, and that
\begin{align}
|(T^{n + 1}|_{E_{n + 1}^1})^{-1}|^{-1} \cdot |P_{n + 1}^1 v| &\leq |T^{n + 1} P_{n + 1}^1 v| \leq |P_{n + 1}^2| \cdot |T_n \circ T^n v| \nonumber \\
& \leq |P_{n +1}^2| \cdot |T_n| \cdot |T^n v| \leq |P_{n + 1}^2| \cdot |T_n| \cdot |T^n|_{F_n^1}| \nonumber \\
\Rightarrow |P^1_{n+1}| &\leq |(T^{n+1}|_{E_{n+1}^1})^{-1}| \cdot |P_{n+1}^2| \cdot |T_n| \cdot |T^n|_{F_n^1}|
\end{align}
Since the bound on $|P_{n+1}^1 v|$ holds over all unit vectors $v \in F_n^1$, we have shown the inequality
\[
\d(F_n^1, F_{n + 1}^1) \leq \sqrt{m} \cdot |T_n| \cdot |T^n|_{F_n^1}| \cdot |(T^{n + 1}|_{E_{n + 1}^1})^{-1}|,
\]
where $\d$ refers to the gap between subspaces; see Section 2.1.2. Using \eqref{eq:minEn1Expansion} and \eqref{eq:Fn1control}, and the hypothesis that $\lim_{n \to \infty} \frac{1}{n} \log^+ |T_n| = 0$, observe that
\begin{align} \label{eq:convergeRate}
\limsup_{n \to \infty} &\frac{1}{n} \log \bigg( \sqrt{m} \cdot |T_n| \cdot |T^n|_{F_n^1}| \cdot |(T^{n + 1}|_{E_{n + 1}^1})^{-1}| \bigg)  \\
& = 0 + \lim_{n \to \infty} \frac{1}{n} \log \frac{V_{m + 1}(T^n)}{V_m(T^n)} + \lim_{n \to \infty} \frac{1}{n} \log \frac{|T^n|^{m-1}}{V_m(T^n)} \nonumber\\
& = \ul - \ol. \nonumber
\end{align}

Let $N_1 \in \N$ be such that for any $n \geq N_1$, we have that the parenthetical quantity in \eqref{eq:convergeRate} is less than or equal to $e^{n (\ul - \ol + \d)}$. Let $N_2 \geq N_1$ be such that $e^{n (\ul - \ol + \d)} \leq \frac{1}{2 m}$; applying Lemma \ref{lem:apertureProps}, we obtain $\d(F_{n + 1}^1, F_n^1) \leq 2 m e^{n (\ul - \ol + \d)}$, hence $d_H(F_n^1, F_{n + 1}^1) \leq 4 m e^{n (\ul - \ol + \d)} $.

The triangle inequality for $d_H$ now implies that for any $ m \geq n \geq N_2$, we have
\begin{align}\label{eq:finalEst}
d_H(F_n^1, F_m^1) \leq 4 m \sum_{l = n}^{\infty} e^{l (\ul - \ol + \d)} \leq \bigg( \frac{4m}{1 - e^{\ul - \ol + \d}} \bigg) e^{n (\ul - \ol + \d)}.
\end{align}
This implies Claim \ref{cla:slowSpaceCauchy}, on adjusting $\d$ to encapsulate the constant term in the parentheses on the RHS of \eqref{eq:finalEst}.
\end{proof}

From Claim \ref{cla:slowSpaceCauchy}, it follows that $\{F_n^1\}_n \subset \Gc^m(V_0)$ is Cauchy in $d_H$. $\Gc^m(V_0)$ is closed and $(\Gc(V_0), d_H)$ is complete (Proposition \ref{prop:grassProps}), and so $\{F_n^1\}$ has a limit $F \in \Gc^m(V_0)$. We now show that $ F = \underline{F}$, hence $\underline{F}$ is $m$-codimensional, by showing that vectors in $F$ have an asymptotic exponential growth rate at most $\ul$, hence $F \subset \underline{F}$, and that vectors in $V_0 \setminus F$ have the exponential asymptotic growth rate $\ol$, hence $\underline{F} \subset F$. 

\smallskip 
\noindent{\bf Showing $\limsup_{n \to \infty} \frac{1}{n} \log |T^n|_F| \leq \ul$.} Let $v \in F$ be a unit vector, and for each $n$, let $v_n \in F_n^1$ be a unit vector for which $|v - v_n| \leq 2 d_H(F_n^1, F)$. Then,
\[
|T^n v| \leq |T^n (v - v_n)| + |T^n v_n| \leq 2 |T^n| \cdot d_H(F_n^1, F) + |T^n|_{F_n^1}|.
\]
This yields
\begin{align*}
\limsup_{n \to \infty} \frac{1}{n} \log |T^n|_F| & \leq \max \big\{\limsup_{n \to \infty} \frac{1}{n} \log \big( |T^n| \cdot d_H(F_n^1, F) \big) , \limsup_{n \to \infty} \frac{1}{n} \log |T^n|_{F_n^1}| \big\} \\ 
& =  \max\{\ol + (\ul - \ol), \ul\} = \ul.
\end{align*}

\smallskip
\noindent{\bf Showing that for any $v \in V_0 \setminus F$, $\lim_{n \to \infty} \frac{1}{n} \log |T^n v| = \ol$.} Fix $v \in V_0 \setminus F$. Without loss, assume $|v| = 1$, and let $\eta = d(v, F)$; since $F$ is closed, $\eta > 0$. Let $n$ be large enough so that $d(v, F_n^1) \geq \eta/2$, and observe that with $v_n = P_n^1 v$, we have $|v_n| \geq d(v, F_n^1) \geq \eta/2$. Therefore,
\begin{align}\label{eq:fastVectorLB}
|T^n v| \geq |T^n v_n| - |T^n (v - v_n)| & \geq |(T^n|_{E_n^1})^{-1}|^{-1} \cdot |v_n| - |T^n|_{F_n^1}| \cdot |I - P_n^1| \nonumber \\
& \geq \frac{\eta}{2} |(T^n|_{E_n^1})^{-1}|^{-1} - 2 |T^n|_F| \cdot | P_n^1|.
\end{align}

By \eqref{eq:minEn1Expansion}, \eqref{eq:Pn1Control} and \eqref{eq:Fn1control}, the growth rate of the RHS of \eqref{eq:fastVectorLB} is at least $\ol$, and so the growth rate of $|T^n v|$ is at least $\ol$. 

For $\eta \in (0,1)$, that the growth on the cone $C_{\eta} = \{v \in V_0 :  d(v, \underline{F}) \geq \eta |v|\}$ is uniform, as in \eqref{eq:uniformGrowth}, is already apparent in the above proof.

Let $E \subset V_0$ be a complement to $\underline{F}$: we now check \eqref{eq:staticVolumeGrowth}. Note that by Lemma \ref{lem:angleForm}, for any $\eta \in (0, |\pi_{E \ds \underline{F}}|^{-1})$, we have $E \subset C_{\eta}$, and so by \eqref{eq:uniformGrowth}, it follows that $\lim_{n \to \infty} \frac{1}{n} \log |(T^n|_E)^{-1}| = - \ol$. Now, by Proposition \ref{prop:approxSVD}, $|(T^n|_E)^{-1}|^{-m} \lesssim \det(T^n|E) \lesssim |T^n|^m$: this yields \eqref{eq:staticVolumeGrowth}.
\end{proof}

\subsubsection{Proof of Lemma \ref{lem:lyapInduct}}

Here we will prove the induction step, Lemma \ref{lem:lyapInduct}. Throughout, we will assume that $T$ is a cocycle satisfying the hypotheses of Lemma \ref{lem:firstSubMET}. We will need the following Lemma, which says that a complement $E$ to $F_2(x)$ does not collapse too quickly onto $F_2(x)$ under the action of $T^n_x$.

\begin{lem}\label{lem:compAngle}
Let $x \in \Gamma$, and let $E \subset \Bc$ be a topological complement to $F_2(x)$. Then,
\begin{align}\label{eq:projAsymp}
\lim_{n \to \infty} \frac{1}{n} \log |\pi_{T^n_x E \ds T^n_x F_2(x)} | = 0.
\end{align}
\end{lem}
\begin{proof}[Proof of Lemma \ref{lem:compAngle}]

Observe that
\[
|\pi_{T^n_x E \ds T^n_x F_2(x)}| = \sup_{e \in E, f \in F_2(x), |e - f| = 1} \frac{|T^n_x e|}{|T^n_x (e -f)|} \leq \sup_{e \in E, f \in F_2(x), |e - f| = 1} \frac{|T^n_x e|}{|T^n_x e| -|T^n_x f|}.
\]
When $|e - f| = 1$, we have $|e| \leq |\pi_{E \ds F_2(x)}|, |f| \leq  1 +  |\pi_{E \ds F_2(x)}|$. Write $p = |\pi_{E \ds F_2(x)}|$. Then,
\begin{align}\label{eq:projControl}
|\pi_{T^n_x E \ds T^n_x F_2(x)}| & \leq  \frac{|T^n_x | \cdot p}{|(T^n_x|_E)^{-1}|^{-1} \cdot p -|T^n_x |_{F_2(x)}| \cdot (1 +  p)},
\end{align}
where $T^n_x|_E$ is considered as a map $E \to T^n_x E$.

Lemma \ref{lem:firstSubMET} implies that $\lim_{n \to \infty} \frac{1}{n} \log |(T^n_x|_E)^{-1}|^{-1} = \l_1$, since $E \subset C_{\eta}(x)$ for any $\eta \in (0, |\pi_{E \ds F}|)$ (Lemma \ref{lem:angleForm}). Therefore, the numerator and denominator on the RHS of \eqref{eq:projControl} both grow at the exponential rate $\ol$, which implies \eqref{eq:projAsymp}.
\end{proof}

\begin{proof}[Proof of Lemma \ref{lem:lyapInduct}]

We will assume that $l_{m_1 + q} > - \infty$. The proof is similar otherwise, and we omit it.

Let $q \in \N$. We will show separately that $\limsup_{n \to \infty} \frac{1}{n} \log V_q(T^n_x|_{F_2(x)}) \leq l_{m_1 + q} - l_{m_1}$ and that $\liminf_{n \to \infty} \frac{1}{n} \log V_q(T^n_x|_{F_2(x)}) \geq l_{m_1 + q} - l_q$.

\smallskip
\noindent{\bf Lower bound on $\liminf_{n \to \infty} \frac{1}{n} \log V_q(T^n_x|_{F_2(x)})$.} For each $n$, let $H_n \subset \Bc$ be an $(m_1 + q)$-dimensional subspace for which $\det(T^n_x | H_n) \geq \frac{1}{2} V_{m_1 + q}(T^n_x)$. Observe that $H_n \cap F_2(x)$ has dimension $\geq q$; let $G_n \subset H_n \cap F_2(x)$ be any $q$-dimensional subspace. Let $J_n$ be a complement to $G_n$ inside $H_n$ of dimension $m_1$ for which $|\pi_{G_n \ds J_n}| \leq \sqrt{q}$ (Lemma \ref{lem:compExist}).

Using Lemma \ref{lem:detSplit}, we now estimate:
\[
\frac{1}{2} V_{m_1 + q}(T^n_x) \leq \det(T^n_x | H_n) \leq C_q |\pi_{G_n \ds J_n}|^q \det(T^n_x | G_n) \cdot \det(T^n_x | J_n) \leq C_q' V_q(T^n_x|_{F_2(x)}) \cdot V_{m_1}(T^n_x)
\]
for constants $C_q, C_q'$ depending on $q$ alone, yielding
\begin{align}\label{eq:lowerGrowBound}
\liminf_{n \to \infty} \frac{1}{n} \log V_q(T^n_x|_{F_2(x)}) \geq l_{m_1 + q} - l_{m_1}.
\end{align}

\smallskip
\noindent{\bf Upper bound on $\limsup_{n \to \infty} \frac{1}{n} \log V_q(T^n_x|_{F_2(x)})$.}  For each $n$, let $G_n \subset F_2(x)$ be a $q$-dimensional subspace for which $\det(T^n_x | G_n) \geq \frac{1}{2} V_q(T^n_x|_{F_2(x)})$. Observe that $T^n_x G_n$ is always $q$-dimensional because $V_q(T^n_x|_{F_2(x)}) > 0$ for any $n$, which follows from \eqref{eq:lowerGrowBound} and the assumption $l_{m_1 + q} > - \infty$.

Let $E \subset \Bc$ be any topological complement to $F_2(x)$, so that $\dim(E \oplus G_n) = m_1 + q$ for any $n$. As $|\pi_{T^n_x E \ds T^n_x G_n}| \leq |\pi_{T^n_x E \ds T^n_x F_2(x)}|$, we obtain
\[
V_{m_1 + 1}(T^n_x) \geq \det(T^n_x | E \oplus G_n) \geq C_{m_1}^{-1} |\pi_{T^n_x E \ds T^n_x F_2(x)}|^{-m_1} \det(T^n_x | E) \det(T^n_x | G_n)
\]
by Lemma \ref{lem:detSplit}, with $C_{m_1}$ depending on $m_1$ alone. We conclude that
\[
l_{m_1 + q} + m_1 \limsup_{n \to \infty} \frac{1}{n} \log |\pi_{T^n_x E \ds T^n_x G_n}| - \liminf_{n \to \infty} \frac{1}{n} \log \det(T^n_x | E) \geq \limsup_{n \to \infty} \frac{1}{n} \log V_q(T^n_x|_{F_2(x)}).
\]
The second term of the LHS equals zero by Lemma \ref{lem:compAngle}, and the $\liminf$ in the third term is a limit, equalling $l_{m_1}$ by \eqref{eq:firstSubVolGrow2}.
\end{proof}

\subsubsection{Proof of the Volume Growth Clause in Theorem \ref{thm:volMET}}

We finish the paper by proving \eqref{eq:correctVolGrow} from Theorem \ref{thm:volMET}.

\begin{lem}\label{lem:rightVolGrow}
For any $x \in \Gamma$ and for any Lyapunov exponent $\l_{i + 1}, i \geq 1$, if $E$ is any complement to $F_{i + 1}(x)$, then
\begin{align}\label{eq:volGrowLemma}
\lim_{n \to \infty} \frac{1}{n} \log \det(T^n_x | E) = \sum_{j = 1}^i m_j \l_j.
\end{align}
\end{lem}
\begin{proof}
The ideas for the proof are already present in the case when $E$ is a complement to $F_3(x)$ (that is, $i = 2$), and so we concentrate on that case.

The first step is to find a splitting of $E$ of the form $E = E_1 \oplus E_2$, where $E_1$ is a complement to $F_2(x)$ in $\Bc$ and $E_2 \subset F_2(x)$. We set $E_2 := F_2(x) \cap E$, and let $E_1$ be a complement to $E_2$ in $E$; it is not hard to check that because $F_2(x) \supset F_3(x)$ and $\Bc = E \oplus F_3(x)$, one has that $\dim E_2 = \dim E - \codim F_2(x) = m_1$, so that $\Bc = E_1 \oplus F_2(x)$ and $F_2(x) = E_2 \oplus F_3(x)$ hold. 

By \eqref{eq:staticVolumeGrowth} in Proposition \ref{prop:staticMET}, we know that for $j = 1,2$,
\[
\lim_{n \to \infty} \frac{1}{n} \log \det(T^n_x | E_j) = m_j \l_j.
\]

We will now use the `block determinant' estimate of Lemma \ref{lem:detSplit}, after first checking that $|\pi_{T^n_x E_1 \ds T^n_x E_2}|$ does not grow too quickly. But $T_x F_2(x) \subset F_2(f x)$ holds for any $x \in \Gamma$, and so $\pi_{T^n_x E_1 \ds T^n_x E_2} = \pi_{T^n_x E_1 \ds F_2(f^n x)}|_{T^n_x E}$. From there, it follows that $\lim_{n \to \infty} \frac{1}{n} \log |\pi_{T^n_x E_1 \ds T^n_x E_2}| = 0$
 by \eqref{eq:projAsymp} in Lemma \ref{lem:compAngle}. Now, by Lemma \ref{lem:detSplit},
\[
\frac{1}{C_{m_1 + m_2} |\pi_{T^n_x E_1 \ds T^n_x E_2}|^{m_1} } \leq \frac{\det(T^n_x | E_1 \oplus E_2)}{\det(T^n_x | E_1) \cdot \det(T^n_x | E_2)} \leq C_{m_1 + m_2} |\pi_{E_1 \ds E_2}|^{m_1} ,
\]
where $C_{m_1 + m_2}$ depends on $m_1 + m_2$ alone; on taking logs and limits, this yields \eqref{eq:volGrowLemma}.
\end{proof}

\begin{center}{\bf  Appendix A: Proof of Proposition \ref{prop:openCond}}\end{center}

\smallskip
Here we record a proof of Proposition \ref{prop:openCond}.
\smallskip

\noindent{\bf Proposition 2.7.} \textit{Let $\Bc = E \oplus F$ be a topological splitting. For any $E' \in \Gc(\Bc)$ for which $d_H(E, E') < |\pi_{E \ds F}|^{-1}$, we have that $\Bc = E' \oplus F$ is a topological splitting.
The projection $\pi_{E' \ds F}$ satisfies}
\begin{align}\label{eq:graphNorm2}
|\pi_{F \ds E'}|_E| \leq \frac{2 d(E, E')}{|\pi_{E \ds F}|^{-1} - d(E, E')}.
\end{align}

\begin{proof}
For simplicity we will write $d(\cdot, \cdot)$ instead of $d_H(\cdot, \cdot)$. To prove that $E'$ complements $F$, we shall show all of the following:
\begin{enumerate}
\item $E' \cap F = \{0\}$ when $d(E, E') < |\pi_{E \ds F}|^{-1}$.
\item The (possibly proper) subspace $E' \oplus F$ is closed when $d(E, E') < |\pi_{E \ds F}|^{-1}$.
\item If the subspace $E' + F$ is closed and $d(E, E') < 1$, then $E' + F = \Bc$.
\end{enumerate}

\noindent {\bf Item 1.} Suppose $E' \cap F \ne \{0\}$. Then for any $e' \in E' \cap F$ with $|e'|=1$,
there exists $e \in E$ with $|e|=1$ such that $|e-e'| < |\pi_{E \ds F}|^{-1}$, hence
\[
1 = |e| = |\pi_{E \ds F} e|  = |\pi_{E \ds F} (e - e')| \leq |\pi_{E \ds F}| \cdot |e - e'| < 1,
\]
which is a contradiction.

\smallskip \noindent
{\bf Item 2.} It will suffice to show that there exists $a > 0$ such that for any $e' \in E', f \in F$, we have
\begin{align} \label{eq:verifyKober}
|e'| \leq a |e' + f|.
\end{align}
This is known as the Kober criterion \cite{Ko}; for a more modern reference, see Chapter VII, \S3, Paragraph (4) of \cite{Day}. Indeed, if \eqref{eq:verifyKober} holds and $x_n = e'_n + f_n$ is a Cauchy sequence, then $e_n'$ and $f_n$ are Cauchy, and thus converge to some $e' \in E', f \in F$ respectively, by the closedness of each of $E', F$, hence with $x := e' + f \in E' + F$, we have $x_n \to x $.

We prove \eqref{eq:verifyKober} directly: Let $e' \in E'$ and $f \in F$.
Fix $\a > 1$ for which $\a d(E, E') < |\pi_{E \ds F}|$, and let $e \in E, |e|  = |e'|$ be 
such that $|e - e'| \leq |e'| \cdot  \a d(E, E')$, so that
\[
|e' + f| \geq |e + f| - |e - e'| \geq |e'| (|\pi_{E \ds F}|^{-1} - \a d(E, E')) =: a^{-1} |e'| > 0 .
\]

\smallskip \noindent
{\bf Item 3.} Seeking a contradiction, assume that $E' \oplus F$ is a proper subspace of $\Bc$. Fix $\a < 1 < \beta$ such that $\beta d(E, E') < \a$ (which we can do, since $d(E, E') < 1$). Since $E' \oplus F$ is closed (by Item 2), the Riesz lemma (Lemma 4.7 of \S 4.1 in \cite{S}) asserts that there exists $x \in \Bc$ with $|x|=1$ such that $|x -  (e' + f)| \geq \a$ for all $e' \in E', f \in F$. On the other hand, since $\Bc = E \oplus F$, we have that $x = e + f$ for some $e \in E, f \in F$. Let $e' \in E'$ be such that $|e - e'| \leq \beta d(E, E')$. Then
\[
|x - (e' + f)| = |e - e'| \leq \beta d(E, E') < \a,
\]
which is a contradiction.

We now prove \eqref{eq:graphNorm2}. First, one can show directly from the Definition \ref{defn:angle} of the minimal angle $\theta$ that $\sin \theta(E', F) \geq \sin \theta(E, F) - d(E, E')$, whence
\begin{align}\label{eq:projectionNormEstimate}
|\pi_{E' \ds F}| \leq \big( |\pi_{E \ds F}|^{-1} - d(E, E')\big)^{-1}
\end{align}
by Lemma \ref{lem:angleForm}. Now, fix $e \in E$ and $\a > 0$. Let $e' \in E', |e'| = 1$ for which $|e - e'| \leq \a \cdot d(E, E')$. Using \eqref{eq:projectionNormEstimate}, we estimate
\[
|\pi_{F \ds E'} e| = |\pi_{F \ds E'} (e - e')| \leq |\pi_{F \ds E}| \cdot |e - e'| \leq 2 \cdot \big( |\pi_{E \ds F}|^{-1} - d(E, E')\big)^{-1} \cdot d(E, E'),
\]
which completes the proof.
\end{proof}

\begin{center}
{\bf Appendix B: Measurability of slow-growing subspaces}
\end{center}

In this appendix, we formulate and prove a measurability result (Lemma \ref{lem:measurabilitySubspace} below) for slow-growing Oseledets subspaces. We note that a potential source of complication is that the Banach space $\Bc$ may not be separable.

\medskip

The following notion of measurability is suitable for our purposes.
\begin{defn}\label{defn:measurability}
Let $(Y, \mathcal H)$ be a measurable space and let $Z$ be a metric space. We say that a mapping $\Phi : Y \to Z$ is \emph{measurable} if $\Phi$ is the pointwise limit of a sequence of finite-valued measurable maps $\Phi_n : Y \to Z$.
\end{defn}
\noindent The limit of a pointwise-convergent sequence of uniformly measurable mappings is itself uniformly measurable (see Chapter III, \S 0 of \cite{castaing2006convex}).

\medskip

In what follows, $(X, \Fc, \mu)$ is an ergodic mpt and $T : X \to L(\Bc)$ is a linear cocycle obeying the hypotheses of Theorem \ref{thm:volMET};  we write $\{\l_i\}$ for the Lyapunov exponents of $T$. We consider the measurable space $(\Gamma, \Fc|_\Gamma)$, where $\Gamma \subset X$ is defined in the proof of Theorem \ref{thm:volMET} and $\Fc|_{\Gamma}$ denotes the restriction of the sigma algebra $\Fc$ to $\Gamma$.

We let $\mathscr C$ denote the set of closed, bounded subsets of $\Bc$, which we regard as a metric space with the Hausdorff metric $d_H$ as defined in Section 2.1.2. We write $B_1 = \{v \in \Bc : |v| \leq 1\}$ for the unit ball of $\Bc$.

\medskip

\begin{lem}\label{lem:measurabilitySubspace}
The mapping $x \mapsto B_1 \cap F_{\l_i}(x)$ is measurable as a map $(\Gamma, \Fc|_{\Gamma}) \to \mathscr C$.
\end{lem}

\begin{proof}
We present the proof in the case $i = 2$; the general case is a straightforward adaptation of the same 
ideas, and is left to the reader.

\medskip

We use the following notation: for $A \in L(\Bc)$ and $c > 0$, we define
\[
\Sc_c(A) := \{ v \in B_1 : |A v| \leq c \} \, .
\]
It is easy to show that for $B \in L(\Bc)$, we have $d_H(\Sc_c(A), \Sc_c(B)) \leq c^{-1} |A - B|$; in particular, if $A_n \to A$ is a convergent sequence of bounded operators in the uniform norm, then $\Sc_c(A_n) \to \Sc_c(A)$ in $d_H$. In light of \eqref{eq:measCondition} it follows that $x \mapsto \Sc_c(T^n_x)$ is measurable in the sense of Definition \ref{defn:measurability} for any $n \geq 1, c > 0$.

To complete the proof, it suffices to show the following.
\begin{lem}\label{lem:sectionConverge}
Let $x \in \Gamma$, $\d \ll \l_1 - \l_2$. Then,
\begin{align}\label{eq:dhLimit}
B_1 \cap F_{\l_2}(x) =  d_H \operatorname{-} \lim_{n \to \infty} \Sc_{e^{n (\l_2 + \d)}}(T^n_x) \, .
\end{align}
Here, $d_H \operatorname{-} \lim$ denotes a limit in the Hausdorff topology on $\mathscr C$.
\end{lem}
\noindent Because the pointwise limit of measurable mappings is measurable, \eqref{eq:dhLimit} implies that $x \mapsto B_1 \cap F_{\l_2}(x)$ is measurable. It remains to prove Lemma \ref{lem:sectionConverge}. \qedhere

\begin{proof}[Proof of Lemma \ref{lem:sectionConverge}] 
Let $E \subset \Bc$ be a complement to $F_{\l_2}(x)$ for which $|\pi_{ F_{\l_2}(x) \ds E}| \leq \sqrt m_1 + 2$ (where $\codim F_{\l_2}(x) = m_1$; see Remark \ref{rmk:exponents}). By Proposition \ref{prop:staticMET}, we have the following for $n$ sufficiently large:
\begin{gather}
\label{eq:slowgrow} |T^n_x|_{F_{\l_2}(x)}| \leq e^{n (\l_2 + \d)} \, , \text{ and } \\
\label{eq:fastgrow} \min_{v \in E \setminus \{0\}} \frac{|T^n_x v|}{|v|} \geq e^{n (\l_1 - \d)} \, .
\end{gather}
We now prove \eqref{eq:dhLimit}. First, note that by \eqref{eq:slowgrow}, for $n$ sufficiently large we have that $B_1 \cap F_{\l_2}(x) \subset \Sc_{e^{n(\l_2  + \d)}}(T^n_x)$. Trivially, then, we have that
\[
\limsup_{n \to \infty} \sup_{v \in B_1 \cap F_{\l_2}(x)} d(v, \Sc_{e^{n(\l_2  + \d)}}(T^n_x)) = 0 \, .
\]

Going the other direction, fix $\e > 0$, and enlarge $n$ so that $(3 + \sqrt m_1) e^{n (\l_2 - \l_1 + 2 \d)} < \e$. Let $w \in \Sc_{e^{n(\l_2  + \d)}}(T^n x)$ and write $w = e + f, e \in E, f \in F_{\l_2}(x)$. Note that $|f| \leq (\sqrt m_1 + 2)$, and so applying \eqref{eq:slowgrow} and \eqref{eq:fastgrow},
\[
e^{n (\l_1 - \d)} |e| \leq |T^n_x e| \leq |T^n_x w| + |T^n_x f| \leq e^{n (\l_2 + \d)}  + (\sqrt m_1 + 2)e^{n (\l_2 + \d)} = (3 + \sqrt m_1) e^{n (\l_2 + \d)} \, ;
\]
it follows that $|e| \leq \e$, and so in fact we have the bound $|f| \leq |w| + |e| \leq 1 + \e$. Thus,
\[
d(w, B_1 \cap F_{\l_2}(x))  \leq |w - \frac{1}{1 + \e} f| \leq |w - f| + \e = |e| + \e \leq 2 \e \, ;
\]
we have shown that
\[
\limsup_{n \to \infty} \sup_{w \in \Sc_{e^{n(\l_2  + \d)}}(T^n_x)} d(w, B_1 \cap F_{\l_2}(x)) \leq 2 \e \, .
\]
Taking $\e \to 0$ completes the proof.
\end{proof}
\end{proof}

\bibliography{biblio}
\bibliographystyle{plain}

\end{document}